\documentclass[a4paper]{amsart}

\usepackage[all]{xy}
\usepackage{amsmath}
\usepackage{amssymb}
\usepackage{amsthm}
\usepackage{latexsym}
\usepackage{enumerate}
\usepackage{prettyref}
\usepackage{hyperref}

\newtheorem{mainthm}{Theorem}

\newtheorem{theorem}{Theorem}[section]
\newrefformat{thm}{\hyperref[{#1}]{Theorem~\ref*{#1}}}
\newtheorem{definition}[theorem]{Definition}
\newrefformat{def}{\hyperref[{#1}]{Definition~\ref*{#1}}}
\newtheorem{construction}[theorem]{Construction}
\newrefformat{const}{\hyperref[{#1}]{Construction~\ref*{#1}}}
\newtheorem{lemma}[theorem]{Lemma}
\newrefformat{lem}{\hyperref[{#1}]{Lemma~\ref*{#1}}}
\newtheorem{proposition}[theorem]{Proposition}
\newrefformat{prop}{\hyperref[{#1}]{Proposition~\ref*{#1}}}
\newtheorem{corollary}[theorem]{Corollary}
\newrefformat{cor}{\hyperref[{#1}]{Corollary~\ref*{#1}}}
\newtheorem{remark}[theorem]{Remark}
\newrefformat{rem}{\hyperref[{#1}]{Remark~\ref*{#1}}}
{
\newtheorem{examplecore}[theorem]{Example}}
\newrefformat{ex}{\hyperref[{#1}]{Example~\ref*{#1}}}

\newcommand{\Ao}{\ensuremath{\mathbb{A}^1}}

\newcommand{\Hom}{\ensuremath{\operatorname{Hom}}}

\newcommand{\im}{\ensuremath{\operatorname{im}}}

\newcommand{\Sing}{\ensuremath{\operatorname{Sing}_{\bullet}^{\Ao}}}
\newcommand{\Spec}{\ensuremath{\operatorname{Spec}}}

\newcommand{\Z}{\mathbb{Z}}
\newcommand{\grsc}[1]{\Z[{#1}^\times/({#1}^\times)^2]}
\newcommand{\sqc}[1]{\left\langle{#1}\right\rangle}
\newcommand{\pf}[1]{\left\langle\!\left\langle{#1}\right\rangle\!\right\rangle}
\newcommand{\rpb}[1]{{\mathcal{RP}}(#1)}
\newcommand{\rrpb}[1]{\widetilde{\mathcal{RP}}(#1)}
\newcommand{\rbl}[1]{{\mathcal{RB}}(#1)}
\newcommand{\rrbl}[1]{{\widetilde{\mathcal{RB}}}(#1)}
\newcommand{\pb}[1]{\mathcal{P}(#1)}
\newcommand{\redpb}[1]{\tilde{\mathcal{P}}(#1)}
\newcommand{\indk}[1]{K_3^{\operatorname{ind}}(#1)}

\newcommand{\F}[1]{\mathbb{F}_{#1}}
\begin{document}

\title{On the third homology of $SL_2$ and weak homotopy invariance}

\author{Kevin Hutchinson and Matthias Wendt}

\date{\today}

\address{Kevin Hutchinson, School of Mathematical Sciences, University
  College Dublin}
\email{kevin.hutchinson@ucd.ie}
\address{Matthias Wendt, Fakult\"at Mathematik,
Universit\"at Duisburg-Essen, Thea-Leymann-Strasse 9, 45127, Essen, Germany}
\email{matthias.wendt@uni-due.de}


\begin{abstract}
The goal of the paper is to achieve - in the special case of the
linear group $SL_2$ - some understanding of the relation
between group homology and its $\Ao$-invariant replacement. We discuss
some of the general properties of the $\Ao$-invariant group homology,
such as stabilization sequences and Grothendieck-Witt module
structures. Together with very precise knowledge about refined Bloch
groups, these methods allow us to deduce that in general there is a
rather large difference between group homology and its $\Ao$-invariant
version. In other words, weak homotopy invariance fails for $SL_2$
over many families of non-algebraically closed fields. 
\end{abstract}

\maketitle
\setcounter{tocdepth}{1}
\tableofcontents

\section{Introduction}
\label{sec:intro}

In this paper, we investigate the difference between group homology
and its $\Ao$-invariant version. It is well-known that K-theory and
hence also the homology of the infinite general linear group $GL$ is
$\Ao$-invariant, i.e., for any regular ring $R$ and any $n\geq 0$, the
map $GL(R)\rightarrow GL(R[T_1,\dots,T_n])$ given by inclusion of
constants induces an isomorphism in homology. This, however, is only a
stable phenomenon: the examples of Krsti{\'c} and McCool
\cite{krstic:mccool} show that $H_1$ of 
$SL_2$ is not $\Ao$-invariant because there exist many non-elementary
matrices in $SL_2(R[T])$ if $R$ is not a field. 

In light of this failure of the strong form of $\Ao$-invariance, one
can combine the separate polynomial rings into the simplicial object
$R[\Delta^\bullet]$ with $n$-simplices
$$R[\Delta^n]=R[T_0,\dots,T_n]/\left(\sum_{i=0}^nT_i=1\right)\cong
R[T_1,\dots,T_n]
$$
and ask if the inclusion of constants $SL_2(R)\to
SL_2(R[\Delta^\bullet])$ induces an isomorphism on group
homology. This is the question of ``weak $\Ao$-invariance'' or ``weak
homotopy invariance''. The attribute ``weak''  is justified by the
fact that if the strong form of $\Ao$-invariance above is satisfied, i.e.,
if the map $SL_2(R)\to SL_2(R[T_1,\dots,T_n])$ induces  an 
isomorphism on homology for all $n\geq 0$, then the spectral sequence computing
the homology of $SL_2(R[\Delta^\bullet])$ collapses and induces an
isomorphim $SL_2(R)\to SL_2(R[\Delta^\bullet])$. More generally, it
makes sense to consider, for an arbitrary linear algebraic group $G$,
the homology of the polynomial singular resolution
$BG(R[\Delta^\bullet])$  as an $\Ao$-invariant  replacement of group
homology, cf. \prettyref{def:ghao} below. As mentioned earlier, weak
homotopy invariance asks if (or when) the natural change-of-topology
morphism $BG(R)\rightarrow BG(R[\Delta^\bullet])$ induces an
isomorphism on homology. In the paper, we will investigate some
properties of this change-of-topology morphism and prove the following 
main result which exhibits situations in which even weak homotopy
invariance fails. 
We refer to \prettyref{thm:main2b}  and \prettyref{thm:main2c} for
precise formulations.  

\begin{mainthm}
\label{thm:main}
\begin{enumerate}
\item Let $k$ be an infinite perfect field of characteristic $\neq 2$
  which is  finitely-generated over its 
  prime field. Let $\ell$ be an odd prime such that
  $[k(\zeta_\ell):k]$ is even,  where $\zeta_\ell$ is a primitive
  $\ell$-th root of unity. Then the kernel of the change-of-topology
  morphism  
$$
H_3(SL_2(k),\mathbb{Z}/\ell)\rightarrow
H_3(BSL_2(k[\Delta^\bullet]),\mathbb{Z}/\ell) 
$$
is not finitely generated. In particular, weak homotopy invariance
with finite coefficients can fail for fields which are not
algebraically closed. 
\item For $k$ a field complete with respect to a discrete valuation,
  with finite residue field $\overline{k}$ of order $q=p^f$ ($p$ odd),
  the change-of-topology morphism  
$$
H_3(SL_2(k),\mathbb{Z}[1/2])\rightarrow
H_3(BSL_2(k[\Delta^\bullet]),\mathbb{Z}[1/2]) 
$$
factors through $\indk{k}\otimes\mathbb{Z}[1/2]$, and its kernel
is isomorphic to the pre-Bloch group
$\pb{\overline{k}}\otimes\mathbb{Z}[1/2]$ which is cyclic of order
$(q+1)'$, where $n'$ denotes the odd part of the positive integer $n$.  
\end{enumerate}
\end{mainthm}

There are several ingredients coming together. On the side of
$\Ao$-invariant group homology, we can use $\Ao$-homotopy theory to
produce stabilization results and Grothendieck-Witt module
structures. On the other hand, the computations in
\cite{hutchinson:rb,hutchinson:bw,hutchinson:loc} 
allow one to understand very explicitly the structure of
$H_3(SL_2(k),\mathbb{Z}[1/2])$. In particular, the above theorem
follows from a comparison of the $k^\times/(k^\times)^2$-module
structures on both sides: the module structure on  the $\Ao$-invariant
group homology factors  
through an action of the Grothendieck-Witt ring of $k$ but the corresponding 
statement for $H_3(SL_2(k),\mathbb{Z}[1/2])$ is false for fields with
non-trivial valuations. 
In the case of number fields there is a further contrast between the two sides:
the $\Ao$-invariant group
homology is finitely generated, while the residue maps of
\cite{hutchinson:rb} show that the group homology itself is not
finitely generated.

Our main theorem also has consequences for the comparison of unstable
K-theories. There are (at least) two natural definitions of unstable
K-theory associated to $SL_2$ and a field $k$, one via Quillen's 
plus-construction definition $\pi_iBSL_2(k)^+$ and one via the
Karoubi-Villamayor definition $\pi_iBSL_2(k[\Delta^\bullet])$. From
\prettyref{thm:main} and the work of Asok and Fasel \cite[Theorem
3]{asok:fasel}, 
we can deduce that these two definitions of unstable K-theory fail to
be equivalent for many families of non-algebraically closed
fields. This is discussed in detail in \prettyref{sec:ktheory}, cf. in
particular \prettyref{prop:k3a} and \prettyref{prop:k3b}.

The methods we use to establish  \prettyref{thm:main} above  only
apply to fields which are not quadratically closed: the 
change-of-topology morphism is injective for 
quadratically closed fields of characteristic $0$,
cf. \prettyref{cor:quadclosed}. Moreover, over algebraically closed
fields of characteristic $0$, 
weak homotopy invariance holds in degree $3$ with finite coefficients,
cf. \prettyref{cor:pos}. Note that weak homotopy invariance with
finite coefficients over algebraically closed fields  is a necessary
ingredient in Morel's approach to the Friedlander-Milnor conjecture
\cite{morel}.    

We would like to mention some questions that could be
pursued in further research: first, the relation between homotopy
invariance and weak homotopy invariance deserves further study. The
behaviour of the spectral sequence associated to the bisimplicial set
$BSL_2(k[\Delta^\bullet])$ plays a major role, and the differentials
of the spectral sequence relate such disparate phenomena like
counterexamples to homotopy invariance, scissors 
congruence groups and surjective stabilization for symplectic
groups. Second, it would be interesting to see what happens to (weak)
homotopy invariance  for homology groups beyond the metastable
range as well as for higher rank groups. 

\emph{Structure of the paper:} We review the definition and basic
properties of the $\Ao$-invariant version of group homology in
\prettyref{sec:hoinvdef}. In \prettyref{sec:stabil} we prove some
stabilization results which are needed in  
\prettyref{sec:indk3} and \prettyref{sec:hoinv}. In \prettyref{sec:indk3} we 
 review the relationship between $H_3(SL_2(k),\Z)$ and $\indk{k}$ and apply these results to the 
change-of-topology morphism in the case of quadratically 
closed fields of characteristic $0$. The Grothendieck-Witt module structures on $\Ao$-invariant homology of $SL_2$ 
are investigated in \prettyref{sec:mod}. Some details on (refined) Bloch
groups and their module structures are provided in
\prettyref{sec:bloch}.  In \prettyref{sec:hoinv}, we use  
the results of the preceding sections to prove our main results:
 that the kernel of the change-of-topology morphism is often very
 large. In \prettyref{sec:ktheory}, we discuss the consequences of our
 results for comparison of unstable K-theories. 
Finally, in \prettyref{sec:cokernel} we conclude with some remarks on
the cokernel of the change-of-topology morphism.  

\emph{Acknowledgements:} We would like to thank Aravind Asok for some
discussions on the computations in
\cite{asok:fasel,asok:fasel:spheres}, and Jens 
Hornbostel and Marco Schlichting for some discussions on finiteness
properties of symplectic K-theory of number fields. 
We are very grateful to the anonymous referee whose detailed report,
helpful remarks and insightful suggestions helped to greatly improve the
paper in form and content.

\section{Group homology made 
\texorpdfstring{$\Ao$}{A1}-invariant: definition} 
\label{sec:hoinvdef}

In this section, we recall the construction which enforces
$\Ao$-invariance in group homology, i.e., replaces group homology by
something representable in $\Ao$-homotopy theory. The crucial
definition is the singular resolution of a linear algebraic group, 
cf. \cite{jardine:homotopy}: 
\begin{definition}
Let $k$ be a field. There is a standard simplicial $k$-algebra
$k[\Delta^\bullet]$ with $n$-simplices given by 
$$
k[\Delta^n]=k[X_0,\dots,X_n]/\left(\sum X_i-1\right)
$$
and face and degeneracy maps given by 
$$
d_i(X_j)=\left\{\begin{array}{ll}
X_j&j<i\\0&j=i\\X_{j-1}&j>i
\end{array}\right.,
\qquad
s_i(X_j)=\left\{\begin{array}{ll}
X_j&j<i\\X_i+X_{i+1}&j=i\\X_{j+1}&j>i
\end{array}\right..
$$
\end{definition}

To a linear algebraic group $G$, we can then associate a simplicial
group $G(k[\Delta^\bullet])$.
Recall that the classifying space of the simplicial group
$G(k[\Delta^\bullet])$ 
is defined to be the diagonal of the bisimplicial set
$BG(k[\Delta^\bullet])$. One simplicial direction is given by the
usual classifying space construction, the other one is given by the
simplicial algebra $k[\Delta^\bullet]$ above. 
We will abuse notation and denote the 
diagonal of the bisimplicial by $BG(k[\Delta^\bullet])$ as well. This
will cause no confusion as (except in \prettyref{sec:cokernel}) we
will never actually deal with the bisimplicial set, only with its
diagonal. 

\begin{definition}\label{def:ghao}
We call the homology of $BG(k[\Delta^\bullet])$ the \emph{group
  homology made $\Ao$-invariant}. There is a natural inclusion
$G(k)\hookrightarrow G(k[\Delta^\bullet])$ which identifies $G(k)$
with the set of zero-simplices $G(k[X_0]/(X_0-1))$ in
$G(k[\Delta^\bullet])$. We refer to this as the
\emph{change-of-topology} morphism. 
\end{definition}

The natural coefficients for the $\Ao$-invariant version of group
homology are $G(k[\Delta^\bullet])$-modules, where a
$G(k[\Delta^\bullet])$-module $M_\bullet$ is a simplicial abelian
group with a simplicial action of  $G(k[\Delta^\bullet])$. While
studying these objects and their relation to continuous
representations of topological groups would be very interesting, we
will only look at trivial coefficients. 

\begin{definition}
\label{def:whoinv}
Let $M$ be an abelian group with trivial $G$-action. Let $M_\bullet$
be the corresponding constant simplicial group, viewed as trivial
$G(k[\Delta^\bullet])$-module. We say that the group 
$G$ has \emph{weak homotopy invariance in degree $n$ over the field
  $k$ with $M$-coefficients}, if
the change-of-topology map induces an isomorphism 
$$
H_n(G(k),M)\stackrel{\cong}{\longrightarrow}
H_n(BG(k[\Delta^\bullet]),M_\bullet). 
$$
\end{definition}

\begin{remark}
The map $G(k)\to G(k[\Delta^\bullet])$ is called change-of-topology
because  we can view $G(k[\Delta^\bullet])$ as the discrete group
$G(k)$ equipped with a topology coming from algebraic
simplices. There is a strong structural similarity between weak
homotopy invariance over an algebraically closed field with finite
coefficients and the Friedlander-Milnor conjecture, cf. \cite[Section
5]{knudson:book}: 
\begin{enumerate}
\item Weak homotopy invariance asks if for an algebraic group $G$, the
  change-of-topology $G(k)\to 
  G(k[\Delta^\bullet])$ (from the discrete group to the group
  topologized by algebraic simplices) induces an isomorphism in
  group homology.
\item Milnor's form of the Friedlander-Milnor conjecture asks if for a
  complex Lie group $G$, the change-of-topology $G^\delta\to G$ (from
  the discrete group to the group with the analytic topology) induces
  an isomorphism in group homology with finite coefficients.
\item Friedlander's generalized isomorphism conjecture asks if for an
  algebraically closed field $k$ and a linear group $G$, the
  change-of-topology from the discrete group $G(k)$ to the group $G$
  ``with the \'etale topology'' induces an isomorphism in group
  homology with finite coefficients.
\end{enumerate}
In fact, it is the main
result of \cite{morel} that the above  questions about different
change-of-topology morphisms are equivalent. Calling the morphism
$G(k)\to G(k[\Delta^\bullet])$ a change-of-topology is 
supposed to underline these similarities.
\end{remark}

The change-of-topology morphism and in particular its effect on third
homology with trivial coefficients is the centre of interest in the
present work.

\section{Group homology made 
\texorpdfstring{$\Ao$}{A1}-invariant: stabilization}
\label{sec:stabil}

In this section, we will develop some stabilization results for
the $\Ao$-invariant homology of special linear and symplectic
groups. These results will be helpful in understanding  the
change-of-topology morphism.  

\subsection{Algebraic groups and fibre bundles in $\Ao$-homotopy
  theory} 

We start with some recollection of statements from $\Ao$-homotopy
theory. Recall that Morel and Voevodsky defined $\Ao$-homotopy theory
in \cite{morel:voevodsky} as the $\Ao$-localization of the category of
simplicial Nisnevich sheaves on the category $\operatorname{Sm}_S$ of
smooth schemes over a base scheme $S$. This paved the way to use
methods of algebraic topology in the study of algebraic varieties. 

As part of an explicit fibrant replacement functor, Morel and
Voevodsky \cite[Section 2.3]{morel:voevodsky} considered the
$\Ao$-singular resolution for simplicial sheaves: to any simplicial
sheaf $X$ on  the site $\operatorname{Sm}_S$ of smooth schemes over
the base scheme $S$, the $\Ao$-singular resolution is given as
$$
\Sing(X)(U)=\Hom(U\times\Delta^\bullet,X).
$$ 
Because a field is a henselian local ring, the simplicial group
$G(k[\Delta^\bullet])$ above is the simplicial group $\Sing(G)(k)$  
of sections over $\Spec k$ of (the Nisnevich sheafification of) the
$\Ao$-singular resolution $\Sing(G)$ of $G$. The same assertion holds
for the corresponding classifying space, i.e., 
$$
(B\Sing(G))(k)\cong B(\Sing(G)(k)).
$$
On the left hand side, we have the sections over $k$ of the
classifying space construction \cite[Section 4.1]{morel:voevodsky}
applied to the simplicial sheaf of groups $\Sing(G)$; on the right
hand side, we have the classifying space associated to the sections of
the simplicial resolution of $G$; these two simplicial sets are weakly
equivalent. The right-hand side space $B(\Sing(G)(k))$ is the one we
are interested in studying because its homology is group homology of
$G$ made $\Ao$-invariant, the left-hand side space $(B\Sing(G))(k)$ is
the one we can study using $\Ao$-homotopy theory.

The stabilization results will be consequences of fibre sequences
associated to principal bundles under algebraic groups. The
corresponding theory was developed by Morel \cite[Sections 8 and
9]{morel:book} and Wendt \cite{torsors}. We only recall the particular
fibre sequences we will need. 

\begin{proposition}
\label{prop:torsors}
\begin{enumerate}
\item 
For each $n\geq 1$, there is a fibre sequence 
$$
\Sing(\mathbb{A}^n\setminus\{0\})(k)\to B\Sing(SL_{n-1})(k)\to
B\Sing(SL_n)(k)
$$
\item 
For each $n\geq 1$, there is a fibre sequence 
$$
\Sing(\mathbb{A}^{2n}\setminus\{0\})(k)\to B\Sing(Sp_{2n-2})(k)\to
B\Sing(Sp_{2n})(k)
$$
\end{enumerate}
\end{proposition}

\begin{proof}
The fibre sequences in the statement are deloopings of the fibre
sequences established in  \cite[Theorem 6.8]{torsors}. For further
details on the proof, cf. loc. cit. Note also that the hypothesis on
the field $k$ being infinite, which appears in \cite{torsors}, is
superfluous in the cases cited because these groups are special in the
sense of Serre.  Alternatively, the first fibre sequence can also be
deduced from \cite[Theorem 8.1, Theorem 8.9]{morel:book}. 
\end{proof}

We remark that the spaces appearing in the above fibre sequences are
$\Ao$-connected so that the choice of base points does not matter.

The next result recalls a consequence of Morel's computation of
low-dimensional homotopy groups of spheres. 

\begin{proposition}
\label{prop:sphere}The simplicial set
$\Sing(\mathbb{A}^{n}\setminus\{0\})(k)$ is $(n-2)$-connected and  for
$n\geq 2$ there is an isomorphism
$$
\pi_{n-1}(\Sing(\mathbb{A}^n\setminus\{0\})(\Spec k))\cong
K^{MW}_n(k), 
$$
where $K^{MW}_n(k)$ denotes the Milnor-Witt K-theory of $k$,
cf. \cite[Section 3]{morel:book}. 
\end{proposition}

\begin{proof}
By \cite[Theorem 6.40]{morel:book}, 
the space $\mathbb{A}^n\setminus\{0\}$ is $\Ao$-$(n-2)$-connected and
there is an isomorphism of Nisnevich sheaves of abelian groups 
$$
\pi_{n-1}^{\Ao}(\mathbb{A}^n\setminus\{0\})\cong K^{MW}_n.
$$
By  \cite[Theorem 8.9]{morel:book}, the space
$\mathbb{A}^n\setminus\{0\}$ is $\Ao$-local and hence there are isomorphisms
$$
\pi_i^{\Ao}(\mathbb{A}^n\setminus\{0\})(k)\cong
\pi_i(\Sing(\mathbb{A}^n\setminus\{0\})(k)).
$$
This implies the claims about the homotopy groups of the simplicial
set. The above argument (with slightly more details) can also be found
in \cite[Proposition 6.6]{torsors}.
\end{proof}

\begin{remark}
\label{rem:sc}
An elementary matrix $e_{ij}(\lambda)\in SL_n(k)$ is in the
path-component of the identity in $SL_n(k[\Delta^\bullet])$, the path
being given by $e_{ij}(\lambda T)\in SL_n(k[T])\cong
SL_n(k[\Delta^1])$. Since $SL_n(k)=E_n(k)$ is generated by such
elementary matrices, we have $\pi_0SL_n(k[\Delta^\bullet])=\{I\}$,
i.e., the simplicial set $SL_n(k[\Delta^\bullet])$ is connected. 
The same is true for $Sp_{2n}(k[\Delta^\bullet])$, or more generally
for the simplicial replacement of any algebraic group whose $k$-points
are generated by unipotent elements, cf. \cite{jardine:homotopy}.

As a consequence, the classifying spaces $BSL_n(k[\Delta^\bullet])$
and $BSp_{2n}(k[\Delta^\bullet])$ are simply connected. 
\end{remark}

\subsection{Stabilization theorems}

We first provide an analogue of the stabilization results of Hutchinson
and Tao \cite{hutchinson:tao} for the homology of $SL_n$ made
$\Ao$-invariant: 

\begin{proposition}
\label{prop:stabsl}
Let $k$ be an infinite field. Then the homomorphisms  
\begin{displaymath}
H_i(BSL_{n-1}(k[\Delta^\bullet]),\mathbb{Z})\rightarrow
H_i(BSL_n(k[\Delta^\bullet]),\mathbb{Z}) 
\end{displaymath}
induced by the standard inclusion $SL_{n-1}(k)\hookrightarrow SL_n(k)$
are isomorphisms for $i\leq n-2$. There is an exact sequence
\begin{displaymath}
H_n(BSL_{n-1}(k[\Delta^\bullet]),\mathbb{Z})\rightarrow
H_n(BSL_n(k[\Delta^\bullet]),\mathbb{Z})  
\rightarrow K^{MW}_n(k)\rightarrow
\end{displaymath}
\begin{displaymath}
\rightarrow 
H_{n-1}(BSL_{n-1}(k[\Delta^\bullet]),\mathbb{Z})\rightarrow
H_{n-1}(BSL_n(k[\Delta^\bullet]),\mathbb{Z})\rightarrow 0.
\end{displaymath}
\end{proposition}

\begin{proof}
By \prettyref{prop:torsors} (1), we have a fibre sequence 
$$
\Sing(\mathbb{A}^n\setminus\{0\})(k)\to B\Sing(SL_{n-1})(k)\to
B\Sing(SL_n)(k).
$$
By \prettyref{rem:sc}, $B\Sing SL_{n-1}(k)$ is simply-connected. 
Since
$$\pi_{n-1}(\Sing(\mathbb{A}^n\setminus\{0\})(k))=
\pi_n(B\Sing(SL_n)(k),B\Sing(SL_{n-1})(k)),
$$
\prettyref{prop:sphere} implies that the pair
$(B\Sing(SL_n)(k),B\Sing(SL_{n-1})(k))$ is $(n-1)$-connected. By the
relative Hurewicz theorem \cite[Corollary III.3.12]{goerss:jardine} we
have a vanishing of relative homology groups of the pair: 
$$
\tilde{H}_i(B\Sing(SL_n)(k),B\Sing(SL_{n-1})(k))=0, i<n
$$
Moreover, the first non-vanishing relative homology group can be
identified as 
\begin{eqnarray*}
&&H_{n}(B\Sing(SL_n)(k),B\Sing (SL_{n-1})(k))\\
&\cong &
\pi_{n}(B\Sing(SL_n)(k),B\Sing (SL_{n-1})(k))\\
&\cong&\pi_{n-1}(\Sing(\mathbb{A}^n\setminus\{0\})(k)).
\end{eqnarray*}
Therefore, using
\prettyref{prop:sphere}, we find that the above relative homology
group $H_{n}(B\Sing(SL_n)(k),B\Sing (SL_{n-1})(k))$ can be
identified with $K^{MW}_n(k)$. The claims then are direct consequences
of the long exact sequence for relative homology of the pair
$(B\Sing(SL_n)(k),B\Sing(SL_{n-1})(k))$. 
\end{proof}

Next, we consider the symplectic groups. This is the main result we
will use in the later development.

\begin{proposition}
\label{prop:stabsp}
Let $k$ be an infinite field. Then the homomorphisms  
\begin{displaymath}
H_i(BSp_{2n-2}(k[\Delta^\bullet]),\mathbb{Z})\rightarrow
H_i(BSp_{2n}(k[\Delta^\bullet]),\mathbb{Z}) 
\end{displaymath}
induced by the standard inclusion $Sp_{2n-2}(k)\hookrightarrow
Sp_{2n}(k)$ are isomorphisms for $i\leq 2n-2$. There is an exact
sequence 
\begin{displaymath}
H_{2n}(BSp_{2n-2}(k[\Delta^\bullet]),\mathbb{Z})\rightarrow
H_{2n}(BSp_{2n}(k[\Delta^\bullet]),\mathbb{Z}) 
\rightarrow K^{MW}_{2n}(k)\rightarrow
\end{displaymath}
\begin{displaymath}
\rightarrow 
H_{2n-1}(BSp_{2n-2}(k[\Delta^\bullet]),\mathbb{Z})\rightarrow
H_{2n-1}(BSp_{2n}(k[\Delta^\bullet]),\mathbb{Z})\rightarrow 0.
\end{displaymath}
\end{proposition}

\begin{proof}
The proof is the same as for \prettyref{prop:stabsl} above, only using
the symplectic versions of \prettyref{prop:torsors} and
\prettyref{rem:sc}. 
\end{proof}

\begin{remark}
Similar stabilization results can be proved for $GL_n$ and some
inclusions of orthogonal groups. 
\end{remark}

\subsection{ $H_{2n}(Sp_{2n}(k),\Z)$ surjects onto $K_{2n}^{MW}(k)$}
For fields $k$ of characteristic not equal to $2$, it is shown in
\cite[Lemma 3.5, Theorem 3.9]{hutchinson:tao} that there is a  natural
homomorphism of graded $\Z[k^\times]$-algebras  
\[
\sigma_n=T_n\circ \epsilon_n:H_n(SL_n(k),\Z)\to K_n^{MW}(k), \quad n\geq 0.
\] 
Here the algebra structure on $\left(H_n(SL_n(k),\Z)\right)_{n\geq 0}$
comes from the external product 
\[
H_n(SL_n(k),\Z)\otimes H_m(SL_m(k),\Z) \to H_{n+m}(SL_{n+m}(k),\Z)) 
\]
induced by the block matrix homomorphism $SL_n(k)\times SL_n(k)\to
SL_{n+m}(k)$. 

Thus, for any $n\geq 0$, the inclusion $Sp_{2n}(k)\to SL_{2n}(k)$
induces a natural map $H_{2n}(Sp_{2n}(k),\Z)\to K_{2n}^{MW}(k)$.

\begin{lemma} 
\label{lem:sp2nkmw}
For any field $k$ of characteristic not equal to $2$,
  the natural map $H_{2n}(Sp_{2n}(k),\Z)\to K_{2n}^{MW}(k)$ is
  surjective. 
\end{lemma}
\begin{proof}
Since $Sp_2(k)=SL_2(k)$, the map $H_2(Sp_2(k),\Z)\to K_2^{MW}(k)$ is
an isomorphism, cf. \cite[Theorem 3.10]{hutchinson:tao}. There are
natural group homomorphisms $Sp_{2n}(k)\times Sp_{2m}(k)\to
Sp_{2n+2m}(k)$, and hence, for any $n$ a homomorphism 
\[
\underbrace{Sp_2(k)\times\cdots\times Sp_2(k)}_n\to Sp_{2n}(k).
\]  
These maps induce a commutative diagram 
\[
\xymatrix{
H_2(Sp_2(k),\Z)\otimes \cdots \otimes H_2(Sp_2(k),\Z)\ar[r]\ar[d]^-{\cong}
&
H_{2n}(Sp_{2n}(k),\Z)\ar[d]\\
H_2(SL_2(k),\Z)\otimes \cdots \otimes H_2(SL_2(k),\Z)\ar[r]\ar[d]^-{\cong}
&
H_{2n}(SL_{2n}(k),\Z)\ar[d]\\
K_2^{MW}(k)\otimes \cdots \otimes K_2^{MW}(k)\ar[r] 
&
K_{2n}^{MW}(k)\\
}
\]

Finally, since $K_n^{MW}(k)$ is additively generated  by
products $[a_1]\cdots[a_n]$, $[a_i]\in K_1^{MW}(k)$, the lowest
horizontal arrow in this diagram is a surjection.  
\end{proof}

\subsection{Injective stabilization for symplectic groups}

We will provide an improvement of the stabilization result
\prettyref{prop:stabsp} for the symplectic groups, using the
surjectivity statement of \prettyref{lem:sp2nkmw}. The main ingredient
is the following comparison between stabilization morphisms for
discrete and simplicial groups. 

\begin{proposition}
\label{prop:surj}
Let $k$ be a field of characteristic $0$. 
\begin{enumerate}
\item 
There is a commutative diagram 
\begin{center}
\begin{minipage}[c]{10cm}
\xymatrix{
H_{n}(SL_{n}(k),\mathbb{Z}) \ar[rr] \ar[rd]_{T_n\circ\epsilon_n} & &
H_{n}(BSL_{n}(k[\Delta^\bullet]),\mathbb{Z}) \ar[ld] \\
& K^{MW}_{n}(k),
}
\end{minipage}
\end{center}
where the top morphism is the change of topology, the left descending
morphism is the one of \cite{hutchinson:tao} and the right descending
morphism is the one from the stabilization sequence of
\prettyref{prop:stabsl}. 
\item There is a commutative diagram 
\begin{center}
\begin{minipage}[c]{10cm}
\xymatrix{
H_{2n}(Sp_{2n}(k),\mathbb{Z}) \ar[rr] \ar[rd]_{T_n\circ\epsilon_n\circ
  \iota} & &
H_{2n}(BSp_{2n}(k[\Delta^\bullet]),\mathbb{Z}) \ar[ld] \\
& K^{MW}_{2n}(k),
}
\end{minipage}
\end{center}
where the top morphism is the change of topology, the left descending
morphism is the one of \prettyref{lem:sp2nkmw} and the right descending
morphism is the one from the stabilization sequence of
\prettyref{prop:stabsp}. 
\end{enumerate}
\end{proposition}

\begin{proof}
We first show that (1) implies (2). 
Consider the following commutative diagram:
\begin{center}
\begin{minipage}[c]{10cm}
\xymatrix{
H_{2n}(Sp_{2n}(k),\mathbb{Z}) \ar[r] \ar[d] &
H_{2n}(BSp_{2n}(k[\Delta^\bullet]),\mathbb{Z}) \ar[r] \ar[d] &
K^{MW}_{2n}(k) \ar[d]^=\\
H_{2n}(SL_{2n}(k),\mathbb{Z}) \ar[r]  &
H_{2n}(BSL_{2n}(k[\Delta^\bullet]),\mathbb{Z}) \ar[r] &
K^{MW}_{2n}(k) 
}
\end{minipage}
\end{center}
The right square is part of a commutative ladder of stabilization
sequences arising as in \cite[Section 3, after Lemma 3.1]{asok:fasel}
from the inclusion $Sp_{2n}\hookrightarrow 
SL_{2n}$ and the subsequent isomorphism of quotients
$SL_{2n}/SL_{2n-1}\cong Sp_{2n}/Sp_{2n-2}$. The left square is simply
induced by the respective inclusions of groups and change-of-topology
maps. Our goal is to show that the top composition
$H_{2n}(Sp_{2n}(k),\mathbb{Z})\rightarrow K^{MW}_{2n}(k)$ is the map
from \prettyref{lem:sp2nkmw}. But this is a consequence of the
commutativity and (1). 

To show (1), we  use the stabilization results above,
in a manner analogous to the arguments in \cite[Section
3]{asok:fasel:spheres}. There is a case distinction, based on the parity of
$n$. In the case $n=2i+1$, we consider the following diagram:
\begin{center}
\begin{minipage}[c]{10cm}
\xymatrix{
& K^{MW}_{2i+2}(k) \ar[d]^\alpha \ar[rd]^0 \\
H_{2i+1}(SL_{2i+1}(k),\mathbb{Z}) \ar[r]_{\iota_{2i+1}} \ar[d]_\cong  & 
H_{2i+1}(BSL_{2i+1}(k[\Delta^\bullet]),\mathbb{Z})\ar[d] 
\ar[r]_{\phantom{sorry}\sigma} & K^{MW}_{2i+1}(k)\\
H_{2i+1}(SL_{2i+2}(k),\mathbb{Z}) \ar[r]_{\iota_{2i+2}} & 
H_{2i+1}(BSL_{2i+2}(k[\Delta^\bullet]),\mathbb{Z}) \ar@{.>}[ru]
}
\end{minipage}
\end{center}
The left vertical map is the stabilization map of
\cite{hutchinson:tao}, which is an isomorphism by \cite[Corollary
6.12]{hutchinson:tao}. The middle column is a part of the
stabilization exact sequence of \prettyref{prop:stabsl}.  The morphism
$\sigma$ is the map we want to compare with $T_n\circ\epsilon_n$. 
Note that the map $\alpha$ is the one from the stabilization
sequence for the inclusion $SL_{2i+1}\hookrightarrow SL_{2i+2}$, while
the morphism $\sigma$ is the one induced from the stabilization
sequence for the inclusion $SL_{2i}\hookrightarrow SL_{2i+1}$. The
composition $\sigma\circ\alpha:K^{MW}_{2i+2}(k)\rightarrow
K^{MW}_{2i+1}(k)$ is then the same as the corresponding composition in
the stabilization sequences for $\Ao$-homotopy groups of \cite[Theorem
6.8]{torsors}. This morphism has been identified as  $0$ in
\cite[Lemma 3.5]{asok:fasel:spheres}, hence $\sigma$ extends through
$H_{2i+1}(BSL_{2i+2}(k[\Delta^\bullet]),\mathbb{Z})$ as claimed. The
question reduces to commutativity of the same diagram for $n=\infty$,
i.e. the stable case.

In the case $n=2i$, we consider the following diagram: 
\begin{center}
\begin{minipage}[c]{10cm}
\xymatrix{
& K^{MW}_{2i+1}(k) \ar[dd]^\alpha \ar@{.>}[rd]^\beta & 0\ar[d]\\
I^{2i+1}(k)\ar@{.>}[rr]^{=\phantom{very sorry}} \ar[d]_\gamma  &  &
I^{2i+1}(k)\ar[d]\\ 
H_{2i}(SL_{2i}(k),\mathbb{Z}) \ar[r]_{\iota_{2i}} \ar[d]  & 
H_{2i}(BSL_{2i}(k[\Delta^\bullet]),\mathbb{Z})\ar[d] 
\ar[r]_{\phantom{sorry}\sigma} & K^{MW}_{2i}(k) \ar[d]\\
H_{2i}(SL_{2i+1}(k),\mathbb{Z}) \ar[r]_{\iota_{2i+1}} \ar[d] & 
H_{2i}(BSL_{2i+1}(k[\Delta^\bullet]),\mathbb{Z}) \ar[d] \ar@{.>}[r]& 
K^{M}_{2i}(k)\ar[d] \\
0&0&0
}
\end{minipage}
\end{center}
The right vertical column is one of the standard exact sequences for
Milnor-Witt K-theory, cf. \cite[Section 5]{morel:puissance}. The left vertical
column is the exact sequence from the Hutchinson-Tao stabilization
theorem \cite[Corollary 6.13]{hutchinson:tao}. The middle column is the
stabilization exact sequence of \prettyref{prop:stabsl}. The diagram
without the dotted arrows is commutative, the only square is obviously
commutative. 

Note that the map $\alpha$ is the one from the stabilization
sequence for the inclusion $SL_{2i}\hookrightarrow SL_{2i+1}$, while
the morphism $\sigma$ is the one induced from the stabilization
sequence for the inclusion $SL_{2i-1}\hookrightarrow SL_{2i}$. The
composition $\sigma\circ\alpha:K^{MW}_{2i+1}(k)\rightarrow
K^{MW}_{2i}(k)$ is then the same as the corresponding composition in
the stabilization sequences for $\Ao$-homotopy groups of \cite[Theorem
6.8]{torsors}. This morphism has been identified as  $\eta$ in
\cite[Lemma 3.5]{asok:fasel:spheres}, hence $\alpha$ factors
through $\beta$ as claimed. 

The composition $\sigma\circ \iota_{2i}\circ\gamma$ has been
identified in \cite[Corollary 6.13]{hutchinson:tao} as the canonical
inclusion $I^{2i+1}(k)\hookrightarrow K^{MW}_{2i}(k)$ arising from
$\eta:K^{MW}_{2i+1}(k)\rightarrow K^{MW}_{2i}(k)$. Hence, we have the
dotted equality arrow in the diagram. This means that on the subgroup
$I^{2i+1}\hookrightarrow H_{2i}(SL_{2i}(k),\mathbb{Z})$, the two maps
- change of topology plus stabilization and $T_{2i}\circ\epsilon_{2i}$
agree. To show that they agree on all of
$H_{2i}(SL_{2i}(k),\mathbb{Z})$, we need to consider the bottom
row. The dotted arrow exists, since $K^{MW}_{2i+1}(k)\rightarrow
K^{MW}_{2i}(k)$ has been identified with $\eta$. 

There is another map $H_{2i}(SL_{2i+1}(k),\mathbb{Z})\rightarrow
K^M_{2i}(k)$ induced from
$T_{2i}\circ\epsilon_{2i}:H_{2i}(SL_{2i}(k),\mathbb{Z})\rightarrow
K^{MW}_{2i}(k)$ modulo $I^{2i+1}$. It now suffices to identify this
map with the bottom composition in the big diagram.

To show that the bottom composition agrees with
$T_{2i}\circ\epsilon_{2i}$ modulo $I^{2i+1}$, it suffices to check
this after stabilization to $SL_\infty$. We have a commutative diagram
\begin{center}
\begin{minipage}[c]{10cm}
\xymatrix{
H_{2i}(SL_{2i+1}(k),\mathbb{Z}) \ar[r] \ar[d]^\cong & 
H_{2i}(BSL_{2i+1}(k[\Delta^\bullet]),\mathbb{Z}) \ar[d]_\cong \\
H_{2i}(SL_\infty(k),\mathbb{Z}) \ar[r]_\cong  & 
H_{2i}(BSL_{\infty}(k[\Delta^\bullet]),\mathbb{Z}) 
}
\end{minipage}
\end{center}
The left isomorphism is a case of \cite[Theorem
1.1]{hutchinson:tao}, the right isomorphism is a case of
\prettyref{prop:stabsl}, and the bottom isomorphism is a consequence
of homotopy invariance of algebraic K-theory.
After stabilization, we see that both maps
$H_{2i}(SL_\infty(k),\mathbb{Z})\rightarrow
K^M_{2i}(k)$  
and $H_{2i}(BSL_{\infty}(k[\Delta^\bullet]),\mathbb{Z})\rightarrow
K^M_{2i}(k)$ factor through Suslin's homomorphism
$H_{2i}(BGL_\infty(k[\Delta^\bullet]),\mathbb{Z})\rightarrow
K^M_{2i}(k)$ and the respective inclusion, cf. also the discussion
after \cite[Lemma 3.1]{asok:fasel}. 
The result is proved.
\end{proof}

The following is now an obvious consequence of
\prettyref{prop:stabsp}, \prettyref{lem:sp2nkmw} and
\prettyref{prop:surj}. 

\begin{corollary}
\label{cor:stabh3}
Let $k$ be a field of characteristic $0$. 
\begin{enumerate}
\item 
In the stabilization sequence for the special linear groups,
cf. \prettyref{prop:stabsl}, the morphism  
$$
H_{2n}(BSL_{2n}(k[\Delta^\bullet]),\mathbb{Z}) 
\rightarrow K^{MW}_{2n}(k)
$$
is surjective. Hence, the standard inclusion $SL_{2n-1}\hookrightarrow
SL_{2n}$ induces an isomorphism 
$$
H_{2n-1}(BSL_{2n-1}(k[\Delta^\bullet]),\mathbb{Z})\stackrel{\cong}{\longrightarrow}
H_{2n-1}(BSL_{2n}(k[\Delta^\bullet]),\mathbb{Z})\cong 
H_{2n-1}(BSL_{\infty}(k)).
$$
\item
In the stabilization sequence for the symplectic groups,
cf. \prettyref{prop:stabsp}, the morphism  
$$
H_{2n}(BSp_{2n}(k[\Delta^\bullet]),\mathbb{Z}) 
\rightarrow K^{MW}_{2n}(k)
$$
is surjective. Hence, the standard inclusion $Sp_{2n-2}\hookrightarrow
Sp_{2n}$ induces an isomorphism 
$$
H_{2n-1}(BSp_{2n-2}(k[\Delta^\bullet]),\mathbb{Z})\stackrel{\cong}{\longrightarrow}
H_{2n-1}(BSp_{2n}(k[\Delta^\bullet]),\mathbb{Z})\cong 
H_{2n-1}(BSp_{\infty}(k)).
$$
\end{enumerate}
\end{corollary}

\begin{remark}
We recall the computation of the homotopy group $\pi_2^{\Ao}(SL_2)$ from
\cite[Theorem 3]{asok:fasel}. For an infinite perfect field of
characteristic $\neq 2$, there is an exact sequence
$$
0\rightarrow T_4'(k)\rightarrow \pi_2^{\Ao}(SL_2)(\Spec
k)\rightarrow KSp_3(k)\rightarrow 0.
$$
The group $T_4'(k)$ sits in an exact sequence
$$
I^5(k)\rightarrow T_4'(k)\rightarrow S_4'(k)\rightarrow 0,
$$
where $I^5(k)$ is the fifth power of the fundamental ideal of the Witt
ring $W(k)$ and there is a surjection $K_4^M(k)/12\twoheadrightarrow
S_4'(k)$.

Comparing with our stabilization result above, we find that the group
$T_4'$ lies in the kernel of the Hurewicz homomorphism
$$
\pi_3^{\Ao}(BSL_2)(k)\cong \pi_3(BSL_2(k[\Delta^\bullet]))\rightarrow 
H_3(BSL_2(k[\Delta^\bullet]),\mathbb{Z}),
$$
at least if the field $k$ has characteristic $0$. 
\end{remark}

\begin{remark}
We want to point out that a stabilization result like
\prettyref{prop:stabsp} can not be true for the discrete groups over
arbitrary fields. Assume that there were  isomorphisms 
$$
H_i(Sp_{2n-2}(k),\mathbb{Z})\cong
H_i(Sp_{2n}(k),\mathbb{Z}) 
$$
for $i\leq 2n-2$ and exact sequences
$$
H_{2n}(Sp_{2n}(k),\mathbb{Z})\to K^{MW}_{2n}(k)\to
H_{2n-1}(Sp_{2n-2}(k),\mathbb{Z})\to
H_{2n-1}(Sp_{2n}(k),\mathbb{Z})\to 0.
$$
Then  by \prettyref{lem:sp2nkmw} $H_3(SL_2(k),\mathbb{Z})\cong
H_3(Sp_\infty(k),\mathbb{Z})$. Together with the stabilization result 
\prettyref{cor:stabh3} above and homotopy invariance for the infinite
symplectic group \cite{karoubi}, this would imply
$H_3(SL_2(k),\mathbb{Z})\cong
H_3(BSL_2(k[\Delta^\bullet]),\mathbb{Z})$. But there are fields for
which  such an isomorphism can not exist, by our
\prettyref{thm:main} which we will prove in the rest of the paper. The
exact form of the optimal 
stabilization result for the symplectic groups seems not yet known.
\end{remark}

\section{The third homology of \texorpdfstring{$SL_2$}{SL2} and
  indecomposable \texorpdfstring{$K_3$}{K3}}\label{sec:indk3} 
Suslin has shown in \cite[Corollary 5.2]{suslin:k3} that for any field
$k$, the Hurewicz homomorphism  
$K_3(k)\to H_3(GL_\infty(k))$ induces an isomorphism
\[
H_3(SL_\infty(k))\cong \frac{K_3(k)}{\{ -1\} \cdot K_2(k)}.
\]
It follows that there is a natural induced surjective map 
\[
H_3(SL_\infty(k))\to \indk{k}:=\frac{K_3(k)}{K_3^M(k)}.
\]
Let $\gamma$ denote the induced composite map
\[
H_3(SL_2(k))\to H_3(SL_\infty(k))\to \indk{k}.
\]
\begin{lemma} 
\label{lem:alpha}
 Let $k$ be an infinite field.
\begin{enumerate}
\item $\gamma$ is surjective.
\item $\gamma$ induces an isomorphism
\[
H_3(SL_2(k),\Z[1/2])\otimes_{\grsc{k}}\Z\cong\indk{k}\otimes\Z[1/2]
\]
\item If $k^\times=(k^\times)^2$ then $\gamma$ induces an isomorphism 
\[
H_3(SL_2(k),\Z)\cong \indk{k}.
\]
\end{enumerate}
\end{lemma}
\begin{proof}
\begin{enumerate}
\item This is \cite[Lemma 5.1]{hutchinson:h3sl}
\item This is \cite[Proposition 6.4 (ii)]{mirzaii:h3gl}.
\item This is \cite[Proposition 6.4 (iii)]{mirzaii:h3gl}.
\end{enumerate}
\end{proof}

We also note that since the map $\gamma$ factors through the stabilization homomorphism we have
\begin{lemma}
\label{lem:bsltoindk}
Let $k$ be a field of characteristic $0$. 
Then there is a natural surjective homomorphism $\gamma':H_3(BSL_2(k[\Delta^\bullet]),\mathbb{Z})\to
\indk{k}$ giving rise to a commutative triangle
\[
\xymatrix{
H_3(SL_2(k),\Z)
\ar[d]\ar[dr]^-{\gamma}
&\\
H_3(BSL_2(k[\Delta^\bullet]),\mathbb{Z})
\ar[r]^-{\gamma'}
&
\indk{k}.}
\]
\end{lemma}

\begin{proof}
There is a natural commutative square 
\[
\xymatrix{
H_3(SL_2(k),\Z)\ar[r]\ar[d]
&
H_3(SL_\infty(k),\Z)\ar[d]^-{\cong}\\
H_3(BSL_2(k[\Delta^\bullet]),\mathbb{Z})
\ar[r]
&
H_3(BSL_\infty(k[\Delta^\bullet]),\mathbb{Z}).
}
\]
In this diagram, the right-hand vertical map is an isomorphism by \prettyref{cor:stabh3}. 
\end{proof}

Combining \prettyref{lem:alpha} (3) and \prettyref{lem:bsltoindk}, we immediately deduce:

\begin{corollary}\label{cor:quadclosed} Let $k$ be a field of characteristic $0$ satisfying $k^\times=(k^\times)^2$.
Then the change-of-topology morphism
\[
H_3(SL_2(k),\Z)\to H_3(BSL_2(k[\Delta^\bullet]),\mathbb{Z})
\]
is injective. The image of this map is isomorphic to $\indk{k}$ and is a direct summand of 
 $H_3(BSL_2(k[\Delta^\bullet]),\mathbb{Z})$. 
\end{corollary}

\begin{corollary}
\label{cor:pos}
Let $k$ be an algebraically closed field of characteristic $0$. Then
weak homotopy invariance holds over $k$ with finite coefficients,
i.e., 
$$
H_3(SL_2(k),\mathbb{Z}/\ell)\cong
H_3(BSL_2(k[\Delta^\bullet]),\mathbb{Z}/\ell)\cong 0.
$$
\end{corollary}

\begin{proof}
By \prettyref{cor:quadclosed}, it suffices to show that
$H_3(BSL_2(k[\Delta^\bullet]),\mathbb{Z}/\ell)\cong 0$. We have
$H_2(BSL_2(k[\Delta^\bullet]),\mathbb{Z})\cong K^{MW}_2(k)$ by
\cite[p.185]{morel:book}. But since 
$k$ is algebraically closed, the group $K^{MW}_2(k)\cong K^M_2(k)$ is
uniquely 
divisible. The universal coefficient theorem then implies that it
suffices to show that
$
H_3(BSL_2(k[\Delta^\bullet]),\mathbb{Z})\otimes_{\mathbb{Z}}\mathbb{Z}/\ell=0. 
$
By \prettyref{cor:stabh3}, it suffices to show vanishing of
$H_3(BSp_\infty(k[\Delta^\bullet]),\mathbb{Z})\otimes_{\mathbb{Z}}\mathbb{Z}/\ell$. 
But then we have the surjection 
$$
KSp_3(k)\otimes_{\mathbb{Z}}\mathbb{Z}/\ell\to
H_3(BSp_\infty(k[\Delta^\bullet])\otimes_{\mathbb{Z}}\mathbb{Z}/\ell 
$$
and unique divisibility of $KSp_3(k)$ implies the required
vanishing. 
\end{proof}

\section{Actions of multiplicative groups and Grothendieck-Witt rings} 
\label{sec:mod}

In this section, we will discuss the comparison of natural
$\mathbb{Z}[k^\times/(k^\times)^2]$-module structures on the homology
groups $H_\bullet(SL_2(k),\mathbb{Z})$ and
$H_\bullet(BSL_2(k[\Delta^\bullet]),\mathbb{Z})$.  
The two main statements are that the change-of-topology morphism is
equivariant for these additional module structures, and that on the
$\Ao$-homotopy side, the
module structure on $H_\bullet(BSL_2(k[\Delta^\bullet]),\mathbb{Z})$
descends to $GW(k)$-module structure. This is achieved by analysing
the action of units on the $\Ao$-homotopy type of $\mathbb{P}^1$.

\subsection{Review of Grothendieck-Witt rings}

We first recall definition and notation for Grothendieck-Witt rings,
collated from various sources \cite{lam:book},
\cite{hutchinson:tao} and \cite[Section 3]{morel:book}.

For a field $k$, we have the ring $\mathbb{Z}[k^\times/(k^\times)^2]$,
which is the integral group ring over the group
$k^\times/(k^\times)^2$ of square classes of $k$.  For $a\in k^\times$
we will let $\sqc{a}\in \Z[k^\times/(k^\times)^2]$ denote the square
class of $a$. We let  $\mathcal{I}_k$ denote the augmentation ideal
with generators $\pf{a}:=\sqc{a}-1$.  

The Grothendieck-Witt ring $GW(k)$ of the field $k$ is the group
completion of the set of isometry classes of nondegenerate symmetric
bilinear forms over $k$. The addition and multiplication operations
are given by orthogonal sum and tensor product of symmetric bilinear
forms, respectively. For $a\in k$, the $1$-dimensional form
$(x,y)\mapsto axy$ is denoted by $\sqc{a}\in GW(k)$. 
The dimension function provides an augmentation
$\dim:GW(k)\rightarrow \mathbb{Z}$, and the augmentation ideal $I(k)$
is called the \emph{fundamental ideal}. It is generated by the Pfister
forms $\pf{a}:=\sqc{a}-1$. 

There are well-known presentations of the Grothendieck-Witt ring of field, dating back to the original 
work of Witt. For example,  Theorem 4.3 in \cite[Chapter II]{lam:book}
 asserts that, as an abelian group, 
$GW(k)$ is generated by the square classes $\sqc{a}$, subject to the additional family of relations
\[
\sqc{a}+\sqc{b}-\sqc{a+b}+\sqc{ab(a+b)}, a,b,a+b\in k^\times.
\]
However, we note that these relations can be re-written (in the group-ring ) as 
\[
\sqc{b}\pf{c}\pf{1-c}
\]
where $c=a/(a+b)$. Thus the additive subgroup generated by these relations is equal to the \emph{ideal} 
of $\Z[k^\times/(k^\times)^2]$ generated by the expressions $\pf{c}\pf{1-c}$. We deduce:

\begin{proposition}\label{prop:gwk}

There is a homomorphism of augmented rings
$$\mathbb{Z}[k^\times/(k^\times)^2]\to GW(k):\sqc{a}\mapsto \sqc{a}.$$
This homomorphism is surjective, and its kernel is the  
ideal, $\mathcal{J}_k$, generated by the `Steinberg elements'
$\pf{a}\pf{1-a}\in \mathcal{I}_k^2$. 
\end{proposition}

Recall from \cite[Definition 3.1]{morel:book} that the Milnor-Witt
K-theory $K^{MW}_\bullet(k)$ of the field $k$ is defined to be the
graded associative ring generated by symbols $[a], a\in k^\times$ of
degree $1$ and $\eta$ of degree $-1$ with the following relations:
\begin{enumerate}
\item $[a][1-a]=0$ for each $a\in k^\times\setminus\{1\}$,
\item $[ab]=[a]+[b]+\eta[a][b]$ for each $a,b\in k^\times$, 
\item $[a]\eta=\eta[a]$ for each $a\in k^\times$,
\item $\eta h=0$ for $h=\eta[-1]+2$.
\end{enumerate}
There is an augmentation $K^{MW}_0(k)\rightarrow
K^{MW}_0(k)/\eta\cong\mathbb{Z}$. The morphism 
$$
GW(k)\to K^{MW}_0(k):\sqc{a}\mapsto 1+\eta[a]
$$
induces an isomorphism of augmented rings. 

\subsection{Module structures and equivariance} 
\begin{construction}
\label{const:action}
We start by describing some actions of the multiplicative group on
$SL_2$ and $\mathbb{P}^1$. 
\begin{description}
\item[An action of $\mathbb{G}_m$ on $SL_2$] recall the standard exact
  sequence of algebraic groups 
$$
1\to SL_2\hookrightarrow
GL_2\stackrel{\det}{\longrightarrow}
\mathbb{G}_m\rightarrow 1.
$$
The determinant $\det:GL_2\to\mathbb{G}_m$ has a splitting given by 
$$
\mathbb{G}_m\to GL_2:u\mapsto \left(\begin{array}{cc}
u&0\\0&1
\end{array}\right).
$$
The splitting followed by the conjugation action of $GL_2$ on $SL_2$
induces an action of $\mathbb{G}_m$ on $SL_2$. Evaluating on
$R$-points for $R$ any $k$-algebra, the action of $k^\times$ on
$SL_2(R)$ is evidently described as   
$$
k^\times\times SL_2(R)\rightarrow SL_2(R): (u,A)\mapsto 
\left(\begin{array}{cc}
u&0\\0&1\end{array}\right)
\cdot A\cdot
\left(\begin{array}{cc}
u^{-1}&0\\0&1\end{array}\right)
$$
From the above description, it is clear that for a $k$-algebra homomorphism
$R\rightarrow S$, the induced homomorphism $SL_2(R)\rightarrow
SL_2(S)$ is $k^\times$-equivariant. 
\item[An action of $k^\times$ on $\mathbb{P}^1(k)$] the projection 
$$
GL_2\rightarrow\mathbb{P}^1:\left(\begin{array}{cc}
a&b\\c&d\end{array}\right)\mapsto [c:d].
$$ 
induces an action of $k^\times$ on $\mathbb{P}^1(k)$. This action is
given by  
$$
k^\times\times\mathbb{P}^1(k)\rightarrow\mathbb{P}^1(k):(u,[x,y])\mapsto
[ux:y]. 
$$
\item[An action of $k^\times$ on ${BSL_2(k[\Delta^\bullet])}$] the  action of
$\mathbb{G}_m$ on $SL_2$ above induces an action of 
$k^\times$ on the simplicial sets $SL_2(k[\Delta^\bullet])$. This
follows from the $k^\times$-equivariance of maps induced from ring
homomorphisms.  This action is an action by group automorphisms (as
opposed to just automorphisms as a space), hence it induces an action
of $k^\times$ on $BSL_2(k[\Delta^\bullet])$. 
\end{description}
\end{construction}

From the above constructions, the following is immediate:

\begin{lemma}
\label{lem:mod1}
The following morphisms are $k^\times$-equivariant, for the module
structures defined above:
\begin{enumerate}[(i)]
\item the Hurewicz homomorphism
  $\pi_\bullet(BSL_2(k[\Delta^\bullet]))\rightarrow
  H_\bullet(BSL_2(k[\Delta^\bullet]))$,
\item the identification $\pi_\bullet(BSL_2(k[\Delta^\bullet]))\cong
  \pi_\bullet^{\Ao}(BSL_2)(k)$,
\item the loop-space isomorphism
  $\pi_{\bullet+1}^{\Ao}(BSL_2)(k)\cong\pi_\bullet^{\Ao}(SL_2)(k)$,
\item the projection 
  $\pi_\bullet^{\Ao}(SL_2)(k)\rightarrow\pi_\bullet^{\Ao}(\mathbb{P}^1)(k)$
  (which happens to be an isomorphism for $\bullet\geq 2$).
\end{enumerate}
\end{lemma}

\begin{proposition}
\label{prop:gw}
Assume that $k$ is an infinite perfect field of characteristic $\neq 2$.
Then the action of $k^\times$ on $SL_2(k[\Delta^\bullet])$ induces a
$GW(k)$-module structure on
$\pi_3(BSL_2(k[\Delta^\bullet]))$ and hence on
$H_3(BSL_2(k[\Delta^\bullet]))$. 
\end{proposition}

\begin{proof}
We explain how the first assertion can be deduced from the results of
\cite{asok:fasel}. Consider the action of $\mathbb{G}_m$ on $BSp_{2n}$
defined in \cite{asok:fasel}, Section $4$, just before Lemma
4.7. Evaluating this action on $\Spec k$ for $n=1$ provides exactly
the action of $k^\times$ on $BSL_2(k[\Delta^\bullet])$ described
above. In \cite[Lemma 4.7, Corollary 4.9]{asok:fasel}, the action of
the sheaf $\mathbb{G}_m$ on the space $BSp_2$ is
described more explicitly: in the notation of loc.cit. there is an
exact sequence 
$$
0\to T_4'(k)\to \pi_2^{\Ao}(Sp_2)(k)\to GW^2_{3}(k)\to 0.
$$ 
Note that  $\pi_2^{\Ao}(Sp_2)(k)\cong \pi_3BSL_2(k[\Delta^\bullet])$
by the work of Moser \cite{moser} and $GW^2_3(k)$  is just different
notation for $KSp_3(k)$. The action of $k^\times$ on $T_4'$ is trivial
by \cite[Corollary 4.9]{asok:fasel}, and the action of $k^\times$ on
$KSp_3(k)$ is such that $u\in k^\times$ acts via multiplication by
$\langle u\rangle\in GW(k)$ on $KSp_3$. In particular, the
computations in \cite[Section 4]{asok:fasel} cited above  imply that
the action of the sheaf $\mathbb{G}_m$ on the 
space $BSL_2$ actually extends to an action of the sheaf
$GW$. Evaluation on $\Spec k$ then implies the first claim.

For the second assertion, recall from \prettyref{rem:sc} that
$BSL_2(k[\Delta^\bullet])$ is simply-connected. Therefore, the 
Hurewicz homomorphism 
$$\pi_3(BSL_2(k[\Delta^\bullet]))\rightarrow
H_3(BSL_2(k[\Delta^\bullet]))$$ is surjective, and by
\prettyref{lem:mod1}, it is also equivariant. 
Therefore, the $GW(k)$-module
structure on $\pi_3$ descends to a
$GW(k)$-module structure on
$H_3$. 
\end{proof}

\begin{lemma}
The conjugation action of $k^\times$ on $SL_2(k)$ descends to an
action of $\mathbb{Z}[k^\times/(k^\times)^2]$ on
$H_\bullet(SL_2(k),\mathbb{Z})$. The natural change-of-topology
morphism  
$$
H_3(SL_2(k),\mathbb{Z})\rightarrow
H_3(BSL_2(k[\Delta^\bullet]),\mathbb{Z}) 
$$
is equivariant for the $\mathbb{Z}[k^\times/(k^\times)^2]$-module
structures. 
\end{lemma}

\begin{proof}
For any unit $u\in k^\times$, the conjugation action of $u^2$ is
the same as conjugating with $\operatorname{diag}(u,u^{-1})\in
SL_2(k)$. The squares therefore act via inner automorphisms, hence
trivially on the homology. Homology groups are abelian groups, so the
action of $k^\times/(k^\times)^2$ can be extended to the group ring
linearly. 

The change-of-topology morphism is induced from the inclusion of
bisimplicial sets $BSL_2(k)\rightarrow BSL_2(k[\Delta^\bullet])$,
where the first bisimplicial set is constant in the
$\Delta^\bullet$-direction. The degree-wise morphisms
$BSL_2(k)\rightarrow BSL_2(k[\Delta^n])$ are induced from the
inclusion of the constants $k\hookrightarrow
k[\Delta^n]$. Equivariance for the $k^\times$-module structures is
then clear. The $\mathbb{Z}[k^\times/(k^\times)^2]$-module structure
on $H_3(SL_2(k))$ has been described above. The
$\mathbb{Z}[k^\times/(k^\times)^2]$-module structure on
$H_3(BSL_2(k[\Delta^\bullet]))$ comes from the $GW(k)$-module
structure in the previous proposition composed with
$\mathbb{Z}[k^\times/(k^\times)^2]\rightarrow GW(k)$. 
The corresponding equivariance is then also clear.
\end{proof}

We state an obvious corollary:

\begin{corollary}
\label{cor:factor}
Assume that $k$ is a infinite perfect field of characteristic $\neq 2$.
We have the following factorization of the change-of-topology
morphism: 
$$H_3(SL_2(k),\mathbb{Z})\rightarrow
H_3(SL_2(k),\mathbb{Z})\otimes_{\mathbb{Z}[k^\times/(k^\times)^2]}GW(k)\to 
H_3(BSL_2(k[\Delta^\bullet]),\mathbb{Z}).$$
\end{corollary}

\section{Bloch groups and specialization homomorphisms}
\label{sec:bloch}

We review the relationship between $H_3(SL_2(k),\Z)$ and the refined
Bloch group, $\mathcal{RB}(k)$ and we use this relationship to compute
lower bounds for the kernel of the map  
\[
H_3(SL_2(k),\Z[1/2]) \to H_3(SL_2(k),\Z[1/2])\otimes_{\Z[k^\times/(k^\times)^2]}GW(k).
\]

We begin by recalling that for any infinite field $k$ there is a natural surjective homomorphism 
$H_3(SL_2(k),\Z)\to \indk{k}$ (\cite[Lemma 5.1]{hutchinson:h3sl}) which induces an isomorphism 
\[
H_3(SL_2(k),\Z[1/2])\otimes_{\grsc{k}}\Z\cong \indk{k}\otimes\Z[1/2]
\]
(\cite{mirzaii:h3gl}).

Recall from \prettyref{sec:mod} that
$\mathcal{I}_k\subseteq\mathbb{Z}[k^\times/(k^\times)^2]$ is the
augmentation ideal and
$\mathcal{J}_k\subseteq\mathbb{Z}[k^\times/(k^\times)^2]$ is the
kernel of $\mathbb{Z}[k^\times/(k^\times)^2]\to GW(k)$. 
Thus, if $M$ is a $\grsc{k}$-module we have 
\[
\mathcal{I}_kM=\ker(M\to M\otimes_{\grsc{k}}\Z)
\]
and
\[
\mathcal{J}_kM=\ker(M\to M\otimes_{\grsc{k}}GW(k)).
\]
In particular, the natural map 
\[
M\otimes_{\grsc{k}}GW(k)\to M\otimes_{\grsc{k}}\Z
\]
(induced by $GW(k)\to \Z$) is an isomorphism if and only if
$\mathcal{J}_kM=\mathcal{I}_kM$. 

For a field $k$ with at least $4$ elements, the \emph{scissors congruence group} or 
\emph{pre-Bloch group}, $\pb{k}$, is the $\Z$-module with generators 
$[a]$, $a\in k^\times$, subject to the relations 
\begin{enumerate}
\item $[1]=0$, and 
\item 
\[
[x]-[y]+\left[\frac{y}{x}\right]-\left[\frac{1-x^{-1}}{1-y^{-1}}\right]+\left[\frac{1-x}{1-y}\right] \mbox{ for }
x,y\not= 1.
\]
\end{enumerate}

The \emph{refined pre-Bloch group}, $\rpb{k}$, is the 
$\Z[k^\times/(k^\times)^2]$-module with generators 
$[a]$, $a\in k^\times$, subject to the relations 
\begin{enumerate}
\item $[1]=0$, and 
\item 
\[
[x]-[y]+\sqc{x}\left[\frac{y}{x}\right]-\sqc{x^{-1}-1}\left[\frac{1-x^{-1}}{1-y^{-1}}\right]+
\sqc{1-x}\left[\frac{1-x}{1-y}\right] \mbox{ for }
x,y\not= 1.
\]
\end{enumerate}

We let 
\[
S_2(k):=\frac{k^\times\otimes_\Z k^\times}{\langle \{ x\otimes y+y\otimes x| x,y\in k^\times\}\rangle},
\]
the second (graded) symmetric power. We let $x\circ y$ denote the image of $x\otimes y$ in $S_2(k)$. We endow 
$S_2(k)$ with the trivial $\Z[k^\times/(k^\times)^2]$-module structure. 

The \emph{refined Bloch group}, $\rbl{k}$, of the field $k$ (with at least $4$ elements) is the kernel of the 
$\grsc{k}$-module homorphism $\Lambda$:
\begin{eqnarray*}
\Lambda=(\lambda_1,\lambda_2):\rpb{k}& \to &\mathcal{I}_k^2\oplus S_2(k),\\
\ [a] & \mapsto & \left( \pf{a}\pf{1-a},a\circ (1-a)\right).\\
\end{eqnarray*}
 
The following is main result (Theorem 4.3 (1)) of \cite{hutchinson:bw}:
\begin{proposition} 
For an infinite field $k$ there is a natural complex 
\[
0\to \mathrm{Tor}(\mu_k,\mu_k)\to H_3(SL_2(k),\Z)\to \rbl{k}\to 0
\]
of $\grsc{k}$-modules which is exact except possibly at the middle term where the homology is annihilated by $4$. 
\end{proposition}

For $1\not= x\in k^\times$, the element $[x]+[1-x]\in \pb{k}$ is independent of $x$ and has order dividing $6$ 
(\cite[Lemma 1.3, Lemma 1.5]{suslin:k3}). We denote this constant
element $[x]+[1-x]$ by $\mathcal{C}_k$. Furthermore, by 
\cite[Lemma 1.2]{suslin:k3}, for 
$x\in k^\times$ the elements $\psi(x):=[x]+[x^{-1}]\in \pb{k}$ satisfy $2\psi(x)=0$ and $\psi(xy)=\psi(x)+\psi(y)$. 
We denote by $\mathcal{S}_k$  the group $\{ \psi(x) | x\in k^\times\}\subset \pb{k}$ and by $\tilde{\mathcal{P}}(k)$ 
the group $\pb{k}/\mathcal{S}_k$. Observe that the natural map $\pb{k}\to \tilde{\mathcal{P}}(k)$ induces an 
isomorphism $\pb{k}\otimes\Z[1/2]\cong \tilde{\mathcal{P}}(k)\otimes\Z[1/2]$.
 
Similarly, we let $\psi_1(x)$ denote the element
$[x]+\sqc{-1}[x^{-1}]\in \rpb{k}$. Note that these elements are not generally 
of finite order. Let $\rrpb{k}$ denote the $\grsc{k}$-module obtained
by taking the quotient of $\rpb{k}$ modulo the  
$\grsc{k}$-module generated by the set $\{ \psi_1(x)| x\in k^\times\}$
and let $\rrbl{k}$ denote the image of the  
composite map $\rbl{k}\to\rpb{k}\to \rrpb{k}$. Then (\cite[Lemma
4.1]{hutchinson:rb}) we have: 
\begin{lemma}\label{lem:rrbl}
The natural map $\rbl{k}\to \rrbl{k}$ is surjective with kernel annihilated by $4$. In particular, 
it induces an isomorphism $\rbl{k}\otimes\Z[1/2]\cong\rrbl{k}\otimes\Z[1/2]$.
\end{lemma}

Now suppose that $k$ is an infinite field with  
surjective valuation $v:k^\times \to \Gamma$, where $\Gamma$ is a totally 
ordered additive abelian group,  and corresponding 
residue field $\bar{k}$. Let $\phi:\Gamma\to \Z/2$ be a group homomorphism. For an abelian 
group $A$, we let $A[\phi]$ denote $A$ endowed with the $\Z[k^\times/(k^\times)^2]$-module structure 
\[
\sqc{x}\cdot a:= (-1)^{\phi(v(x))} a \quad \mbox{for all } x\in k^\times, a\in A.
\]

Then we have (\cite[section 4.3]{hutchinson:bw}): 
\begin{proposition}\label{prop:sv}
There is a natural surjective $\grsc{k}$-module 
homomorphism $S_{v,\phi}=S_{\phi}:\rrpb{k}\to \redpb{\bar{k}}[\phi]$ determined by the formula
\[
S_{\phi}([a])=
\left\{
\begin{array}{ll}
[\bar{a}],& v(a)=0\\
\mathcal{C}_{\bar{k}},& v(a)>0\\
-\mathcal{C}_{\bar{k}},& v(a)<0\\
\end{array}
\right.
\] 

Furthermore, \emph{if $\phi\not=0$} the image of the induced composite homomorphism 
\[
H_3(SL_2(k),\Z)\to  \rbl{k}\to \redpb{\bar{k}}[\phi]
\]
 contains $4\cdot\redpb{\bar{k}}$.
\end{proposition}

The following corollary, which follows from the case $\phi=0$ in \ref{prop:sv}, will be needed below:

\begin{corollary}\label{cor:pb}
Let $k$ be a field with valuation and corresponding residue field $\bar{k}$. 
There is a natural surjective homomorphism $\pb{k}\to \redpb{\bar{k}}$.
\end{corollary}
\begin{proof}
When $\phi=0$, $\redpb{\bar{k}}$ has the trivial $\Z[k^\times/(k^\times)^2]$-module structure and hence the 
homomorphism $S_{\phi}$ factors through $\rpb{k}_{k^\times}=\pb{k}$.  
\end{proof}

\begin{corollary}\label{cor:val}
Let $k$ be a field with surjective valuation $v:k^\times\to \Gamma$ 
and residue field $\bar{k}$.  Suppose that
\begin{enumerate}
\item $\Gamma/2\Gamma\not = 0$ and 
\item $16\cdot\redpb{\bar{k}}\not=0$.
\end{enumerate}
Then $\mathcal{J}_kH_3(SL_2(k),\Z)\not=0$.
\end{corollary}
\begin{proof} Let $\phi:\Gamma\to \Z/2$ be a non-zero homomorphism. Let $y\in \redpb{\bar{k}}$ with $16y\not= 0$. 
There exists $x\in H_3(SL_2(k),\Z)$ with $S_\phi(x)=4y$. 

Choose $\pi \in \mathcal{O}_v$ with $\phi(v(\pi))=1$. 
So $\sqc{\pi}y=-y$ and hence $\pf{\pi}y=\pf{\pi^{-1}}y=-2y$. 

But $v(1-\pi)=0=v(\pi-1)$ and hence 
$v(1-\pi^{-1})=v(\pi^{-1})$ and $\phi(v(1-\pi^{-1}))=1$ also.  Thus $\pf{1-\pi^{-1}}y=-2y$ also. It follows that 
\[
S_\phi(\pf{\pi^{-1}}\pf{1-\pi^{-1}}x)=\pf{\pi^{-1}}\pf{1-\pi^{-1}}S_\phi(x)=(-2)\cdot(-2)\cdot 4y=16y\not=0
\]
and hence $0\not= \pf{\pi^{-1}}\pf{1-\pi^{-1}}x\in \mathcal{J}_kH_3(SL_2(k),\Z)$.
\end{proof}

The following is Theorem 6.19 in \cite{hutchinson:rb}:

\begin{proposition}\label{prop:loc}
Let $k$ be a local field with finite residue field $\bar{k}$ of odd order. If $\mathbb{Q}_3\subset k$, suppose that 
$[k:\mathbb{Q}_3]$ is odd. Then there is an isomorphism of $\grsc{k}$-modules 
\[
H_3(SL_2(k),\Z[1/2])\cong \left( \indk{k}\otimes\Z[1/2]\right)\oplus \left( \pb{\bar{k}}\otimes\Z[1/2]\right).
\]

In this isomorphism, the map from $H_3(SL_2(k),\Z[1/2])$ to the second factor is induced by $S_\phi$ where 
$\phi$ is the nontrivial homomorphism $\Gamma=\Z\to \Z/2$.  

\end{proposition}

\begin{remark}\label{rem:pbfin} 
If $k$ is a finite field with $q$ elements, then $\pb{k}$ has order $q+1$ and $\pb{k}\otimes\Z[1/2]$ 
is cyclic (\cite[Lemma 7.4]{hutchinson:bw}) of order $(q+1)'$. Here, $n'$ denotes the odd part of the integer 
$n$: $n=2^an'$ with $a\geq 0$ and $n'$ odd. 
\end{remark}

More generally we have the following (\cite[Theorem 6.47]{hutchinson:loc}):

\begin{proposition}\label{prop:hloc}
Let $k_0,k_1,\ldots,k_n=k$ be a sequence of fields satisfying:
\begin{enumerate}
\item For each $i\in \{1,\ldots, n\}$ there is a complete discrete value $v_i$ on $k_i$ with residue field
$k_{i-1}$.
\item $k_0$ is either finite or real-closed or quadratically closed.
\item $\mathrm{char}(k_0)\not= 2$.
\item Either $\mathrm{char}(k)=3$ or $\mathrm{char}(k_0)\not=3$ or $k$ contains 
a primitive cube root of unity. 
\end{enumerate}
Then there is a natural split short exact sequence
\[
0\to 
\bigoplus_{i=0}^{n-1}(\pb{k_i}\otimes\Z[1/2])^{\oplus 2^{n-i-1}}
\to
H_3(SL_2(k),\Z[1/2])
\to
\indk{k}\otimes\Z[1/2]
\to 0.
\]
\end{proposition}

\begin{remark}
The direct sum decomposition occurring here is the eigenspace decomposition for the group of characters 
on $k^\times/(k^\times)^2$ which restrict to the trivial character on $k_0^\times/(k_0^\times)^2$ (see 
\cite[section 6]{hutchinson:loc}).

To spell this out, let $k$ be complete with respect to a discrete
valuation $v$ with residue field $\bar{k}$ of  
characteristic not equal to $2$. Let $U:=\{ a\in k^\times|\ v(a)=0\}$. By Hensel's Lemma $u\in U$ is square if 
and only if $\bar{u}\in \bar{k}^\times$ is a square. Thus $U/U^2\cong \bar{k}^\times/(\bar{k}^\times)^2$, and if 
$\pi\in k^\times$ is a uniformizer there is a natural (split) short exact sequence 
 \[
1\to \bar{k}^\times/(\bar{k}^\times)^2\to k^\times/(k^\times)^2\to \pi^{\Z/2}\to 1
\]
where the first injection is obtained by choosing an inverse image $x$ in $U$ of a given element $\bar{x}\in \bar{k}$.

Now let $\mathcal{X}_k:=\mathrm{Hom}(k^\times/(k^\times)^2,\mu_2)$. As noted, the conditions on $k$ in the proposition 
(completeness of $v_i$ and  $\mathrm{char}(k_0)\not=2$) ensure that there are natural injective maps 
\[
k_{i-1}^\times/(k_{i-1}^\times)^2\to k_{i}^\times/(k_{i}^\times)^2
\] 
and hence there are surjective restriction homomorphisms 
\[
\mathcal{X}_{k_i} \to \mathcal{X}_{k_{i-1}}.
\]
For each $i\leq n$, let 
\[
W_i:=\ker(\mathcal{X}_{k} \to \mathcal{X}_{k_{i}})=\{ \chi\in \mathcal{X}_k |\  \chi|_{k_i^\times}=1\}.
\]
 For each $i <n$ and $\chi\in W_i\setminus W_{i+1}$, 
the $\chi$-eigenspace of $H_3(SL_2(k),\Z[1/2])$ -- which we will denote $H_3(SL_2(k),\Z[1/2])_\chi$ -- 
is isomorphic to $\pb{k_i}\otimes\Z[1/2]$, and, by definition, the square class 
$\sqc{a}$ acts as multiplication by $\chi(a)$ on this factor. 
\end{remark}
\begin{lemma}\label{lem:mchi}
Let $k$ be as in Proposition \ref{prop:hloc}. Let $M$ be a $\grsc{k}\otimes\Z[1/2]$-module. Let $i<n$ and 
$\chi\in W_i\setminus W_{i+1}$. Then $\mathcal{J}_kM_\chi=M_\chi$. 
\end{lemma}
\begin{proof}
Let $\mathcal{O}_n=\{ a\in k=k_n| v_n(a)\geq 0\}$ and for $j<n$ define recursively 
$\mathcal{O}_{j}=\{ a\in \mathcal{O}_{j+1}| v_j(\pi_{j}(a))\geq 0\}$, where $\pi_j:\mathcal{O}_{j+1}\to k_j$ is the 
natural surjection.   Let $x\in \mathcal{O}_{i+1}$ with $v_{i+1}(\pi_{i+1}(x))=1$. Then the group (of order $2$) 
$k_{i+1}^\times/\left((k_{i+1}^\times)^2\cdot k_i^\times\right)$ is generated by the class of $x$ and hence 
$\chi(\sqc{x})=\chi(\sqc{x^{-1}})=-1$. Since 
\[
v_{i+1}(\pi_{i+1}(x-1))=0,
\]
 it follows that the class of 
$x-1$ represents an element - possibly trivial - of $k_i^\times/(k_i^\times)^2$ and 
hence $\chi(\sqc{x-1})=1$. Hence 
\[
\chi(\sqc{1-x^{-1}})=\chi(\sqc{x^{-1}})\chi(\sqc{x-1})=-1\cdot 1 =-1.
\]
Thus $\pf{x^{-1}}$ and $\pf{1-x^{-1}}$ both act on $M_\chi$ as multiplication by $-2$, and hence 
$\pf{x^{-1}}\pf{1-x^{-1}}\in \mathcal{J}_k$ acts on $M_\chi$ as multiplication by $4$. 
 \end{proof}
\begin{corollary}
\label{cor:hloc}
If $k$ is as in Proposition \ref{prop:hloc} then 
\begin{enumerate}
\item
\begin{eqnarray*}
H_3(SL_2(k),\Z[1/2])\otimes_{\grsc{k}}GW(k)&\cong&
H_3(SL_2(k),\Z[1/2])\otimes_{\grsc{k}}\Z\\&\cong&\indk{k}\otimes\Z[1/2] 
\end{eqnarray*}
and 
\item
there is a natural short exact sequence 
\[
0\to 
\bigoplus_{i=0}^{n-1}(\pb{k_i}\otimes\Z[1/2])^{\oplus 2^{n-i-1}}
\to
H_3(SL_2(k),\Z[1/2])\to\phantom{bitofspace}
\]
\[
\phantom{some more space}
\to
H_3(SL_2(k),\Z[1/2])\otimes_{\grsc{k}}GW(k)
\to 0.
\]
\item $\mathcal{J}_kH_3(SL_2(k),\Z[1/2])\cong \bigoplus_{i=0}^{n-1}(\pb{k_i}\otimes\Z[1/2])^{\oplus 2^{n-i-1}}$.
\end{enumerate}
\end{corollary}

\begin{proof}
The second statement follows from the first by Proposition \ref{prop:hloc}, and the third is an immediate consequence 
of the second. 

To prove the first isomorphism of  statement (1), we must show that 
$$
\mathcal{J}_kH_3(SL_2(k),\Z[1/2])=\mathcal{I}_kH_3(SL_2(k),\Z[1/2]).
$$
Now 
\begin{eqnarray*}
\mathcal{I}_kH_3(SL_2(k),\Z[1/2]) &=&
\bigoplus_{i=0}^{n-1}(\pb{k_i}\otimes\Z[1/2])^{\oplus 2^{n-i-1}}\\
&=& \bigoplus_{i=0}^{n-1}\left(\bigoplus_{\chi\in W_i\setminus W_{i+1}}H_3(SL_2(k),\Z[1/2])_\chi\right)\\
\end{eqnarray*}
by Proposition \ref{prop:hloc} and the remark which follows it. The result follows by Lemma \ref{lem:mchi}.
\end{proof}

\begin{remark} 
Let $k$ satisfy the hypotheses of Proposition \ref{prop:hloc}. If
$k_0$ is finite or quadratically closed then the Witt ring $W(k_0)$ of
$k_0$ is $2$-torsion. An easy induction using Springer's Theorem on
Witt rings of fields complete with respect to a discrete valuation
(\cite[Chapter VI, Theorem 1.4]{lam:book}) implies that $W(k)$ is
$2$-torsion and hence that $GW(k)\otimes \Z[1/2]=\Z[1/2]$.   

However, in the case that $k_0$ is real closed, then -- by Springer's
Theorem again -- the fundamental ideal of $W(k)$ contains a free
abelian group of rank $2^n$ and the map $GW(k)\otimes \Z[1/2]\to
\Z[1/2]$ has a large kernel.   
\end{remark}

The following is a special case of \cite[Theorem 5.1]{hutchinson:rb}:

\begin{proposition} \label{prop:ufd}
Let $\mathcal{O}$ be a unique factorization domain with field of
fractions $k$. Let $P$ be a set of representatives of the 
orbits  of prime elements of $k$ under the action, by multiplication, of $\mathcal{O}^\times$. 
For each $p\in P$ there is a
discrete valuation $v_p:k^\times\to \Z$ with corresponding residue
field $\bar{k}_p$. Let $\phi:\Z\to \Z/2$ be the non-zero homomorphism.   

Then the specialization homomorphisms induce a well-defined
\emph{surjective} map of $\grsc{k}$-modules  
\[
S=\sum_{p\in P}S_{p,\phi}: 
H_3(SL_2(k),\Z[1/2])\to \bigoplus_{p\in P}\pb{\bar{k}_{{p}}}\otimes\Z[1/2].
\]
\end{proposition}

\begin{corollary}\label{cor:ufd}
 Let $\mathcal{O}$ be a unique factorization domain with field of fractions $k$. 
  Then the map induced by $S$
\[
\mathcal{J}_kH_3(SL_2(k),\Z[1/2])\to \bigoplus_{p\in P}\pb{\bar{k}_{{p}}}\otimes\Z[1/2]
\]
is surjective.

More generally, the map
\[
\mathcal{J}_k(H_3(SL_2(k),\Z)\otimes A)\to \bigoplus_{p\in P}(\pb{\bar{k}_{{p}}}\otimes A)
\]
is surjective for any commutative $\Z[1/2]$-algebra $A$.
\end{corollary}
\begin{proof}
We will use the notation  
$$\mathcal{P}(\mathcal{O}):=\bigoplus_{p\in
  P}\pb{\bar{k}_{{p}}}\otimes\Z[1/2].
$$
 By \prettyref{prop:ufd},  
it is enough to show that $\mathcal{J}_k\mathcal{P}(\mathcal{O})=\mathcal{P}(\mathcal{O})$. 

Let $x\in \mathcal{P}(\mathcal{O})$. There exist 
primes ${p}_1,\ldots,{p}_t\in P$  such that $x=\sum_{i=1}^tx_i$ with 
$x_i\in \pb{\bar{k}_{{p_i}}}\otimes\Z[1/2]$. 

Recall that $\sqc{a}\in \grsc{k}$ acts as multiplication by $(-1)^{v_{\mathfrak{p}_i}(a)}$ on $\pb{\bar{k}_{{p_i}}}$. 
Choose $a\in \mathcal{O}$ with the property that $v_{{p}_i}(a)$ is odd for $1\leq i\leq t$. 
Let $b=1/a$. 
Then $v_{{p}_i}(b)=v_{{p}_i}(1-b)$ is odd for all $i$. It follows that 
$\pf{b}\pf{1-b}x_i=(-2)^2x_i=4x_i$ for all $i$ and hence that 
\[
x=\pf{b}\pf{1-b}\left(\frac{x}{4}\right)\in \mathcal{J}_k\mathcal{P}(\mathcal{O}).
\]
\end{proof}

\begin{remark} With a little more care, one can show that the image of the map
\[
\mathcal{J}_kH_3(SL_2(k),\Z)\to \bigoplus_{p\in P}\redpb{\bar{k}_{{p}}}
\]
contains $\bigoplus_{p\in P}16\cdot\redpb{\bar{k}_{{p}}}$.
\end{remark}

\section{On the failure of weak homotopy invariance}
\label{sec:hoinv}

This section sums up our insights into the failure of weak
homotopy invariance for the third homology of $SL_2$.

\subsection{Module structures}
We show how the failure of weak homotopy invariance derives from the
fact that the square class module structure on $\Ao$-invariant group
homology descends to a Grothendieck-Witt module structure, by the results of
\prettyref{sec:mod}, while the one on group homology typically does
not, by the results of  \prettyref{sec:bloch}.

We begin by noting that the kernel of the change-of-topology morphism is nontrivial for a quite 
general class of fields.
\begin{theorem}
\label{thm:main2a}
Let $k$ be a field of characteristic $\neq 2$, with surjective valuation
$v:k^\times\to \Gamma$ and residue field $\bar{k}$.  Suppose that 
\begin{enumerate}
\item $\Gamma/2\Gamma\not = 0$ and 
\item $16\cdot\redpb{\bar{k}}\not=0$.
\end{enumerate}
Then the kernel of 
the natural
change-of-topology  morphism 
$$
H_3(SL_2(k),\mathbb{Z})\rightarrow
H_3(BSL_2(k[\Delta^\bullet]),\mathbb{Z})
$$ 
is not trivial.

Furthermore, if $\ell$ is an odd prime and if $\pb{\bar{k}}\otimes \Z/\ell\not= 0$, then the  
same statement holds with $\Z$ replaced by $\Z/\ell$. 
\end{theorem}
\begin{proof}
The first statement is a direct consequence of Corollaries \ref{cor:factor} and \ref{cor:val} above. 

For the second statement, the universal coefficient theorem together
with the snake lemma
implies a commutative diagram with exact rows
\begin{center}
\begin{minipage}[c]{10cm}
\xymatrix{
 0 \ar[r] & \ker\phi \ar[r] \ar[d] & \ker\psi \ar[d] \\
0 \ar[r] & H_3(SL_2(k),\mathbb{Z})\otimes \Z/\ell \ar[r] \ar[d]^\phi &
H_3(SL_2(k),\mathbb{Z}/\ell) \ar[d]^\psi \\
0 \ar[r] & H_3(BSL_2(k[\Delta^\bullet]),\mathbb{Z})\otimes\Z/\ell
\ar[r] & H_3(BSL_2(k[\Delta^\bullet]),\mathbb{Z}/\ell)
}
\end{minipage}
\end{center}
But the group $\ker \phi$ contains
$\mathcal{J}_k(H_3(SL_2(k),\Z)\otimes\Z/\ell)$, since  
$GW(k)\otimes \Z/\ell= (\grsc{k}\otimes\Z/\ell)/\mathrm{Image}(\mathcal{J}_k)$. 
This, in turn,  maps onto the nonzero group 
$\pb{\bar{k}}\otimes\Z/\ell=\redpb{\bar{k}}\otimes\Z/\ell$, thus
proving the second claim.
\end{proof}

Next, we observe (\prettyref{thm:main2b} below) 
that when the field $k$ is small, the kernel of the change-of-topology morphism is always large.

\begin{lemma}\label{lem:cheb}
Let $k$ be a global field. Let $\ell$ be an odd prime such that
$[k(\zeta_\ell):k]$ is even, where $\zeta_\ell$ denotes a  primitive
$\ell$-th root of unity. Then there are infinitely many finite places
$v$ of $k$ satisying $\ell|q_v+1$, where  $q_v$ is the cardinality of
the residue field, $\bar{k}_v$, at $v$.  
\end{lemma}
\begin{proof}
Let $L=k(\zeta_\ell)$. By the Chebotarev density theorem 
(see \cite[Chapter VII, Theorem 13.4]{neukirch:book} for number fields and 
\cite[Theorem 9.13A]{rosen:book} for 
global function fields) there are
infinitely many primes of $k$ (not dividing $\ell$) whose  
Frobenius in $\mathrm{Gal}(L/k)$ has order $2$.  It follows that, for such a prime $v$, 
$\zeta_\ell\not\in \bar{k}_v=\F{q_v}$, but $\zeta_\ell\in \F{q_v^2}$. Thus, for such $v$, 
 $\ell\nmid q_v-1$, but $\ell|q_v^2-1$. 
\end{proof}
\begin{remark} Of course, for any given number field $k$, $[k(\zeta_\ell):k]=\ell-1$ for all but finitely many 
odd primes $\ell$. 

If $k$ is a global field of positive characteristic, there are
infinitely many odd primes $\ell$ such that $[k(\zeta_\ell):k]$ is even.  
\end{remark}
\begin{theorem}
\label{thm:main2b}
Let $k$ be an infinite perfect field which is finitely generated over its
prime field, and assume that $\operatorname{char}k\neq 2$. Then the
kernel of the natural change-of-topology  morphism 
$$
H_3(SL_2(k),\mathbb{Z}[1/2])\rightarrow
H_3(BSL_2(k[\Delta^\bullet]),\mathbb{Z}[1/2])
$$ 
is not finitely-generated.

Furthermore, if $\ell$ is an odd prime for which $[k(\zeta_\ell):k]$ is even, then the  
same statement holds with $\Z[1/2]$ replaced by $\Z/\ell$. 
\end{theorem}
\begin{proof} Let $A=\Z[1/2]$ or
$A=\Z/\ell$ for an odd prime $\ell$.

Now the field must contain a subfield, $k_0$ say, isomorphic either to $\mathbb{Q}$ or to $\mathbb{F}_p(x)$ where 
$p=\mathrm{char}(k)>0$.  Let $d$ denote the transcendence degree of $k$ over $k_0$. 
We will prove the result, together with the statement that $\pb{k}\otimes A$ is not finitely generated, by induction on 
$d$.

Suppose first that $d=0$. Then $k$ is a global field. In this case
(fixing an infinite prime in the function field case), the ring of
integers $\mathcal{O}_k$ is a Dedekind domain with finite class group.
Hence there exists $a\in \mathcal{O}_k$ for which
$\mathcal{O}:=\mathcal{O}_k[a^{-1}]$ is a  unique factorization
domain. Now for any prime $p$ of $\mathcal{O}$, $\pb{\bar{k}_p}\otimes
\Z[1/2]$ is cyclic of order $(q_p+1)'$ where $n'$ denotes the
prime-to-$2$ part of the number $n$ (Remark \ref{rem:pbfin}). This,  
together with Lemma \ref{lem:cheb} implies that $\pb{\bar{k}_p}\otimes
A$ is nonzero for infinitely many primes $p$. The claim then follows
from \prettyref{cor:factor} and \prettyref{cor:ufd}.

We also observe that for any field $k$, there is an exact sequence 
\[
\pb{k}\to S_2(k)\to K_2(k)\to 0
\]
where the first map sends $[a]$ to $a\circ (1-a)$ and the second sends $a\circ b$ to 
$\{ a,b\}$.
However when $k$ is a global field $K_2(k)$ is a torsion group while $S_2(k)$ modulo torsion is a 
free abelian group of infinite rank. It follows that 
$\pb{k}$ maps onto a free abelian group of infinite rank, 
and hence $\pb{k}\otimes A$ is not finitely generated. 
 
Now suppose $d>0$ and the result is known for fields of smaller
transcendence degree over $k_0$. Then for any discrete valuation $p$ on $k$, $\bar{k}_p$ is
an infinite finitely-generated field of smaller transcendence
degree. By the proof of \prettyref{cor:val} the kernel of the change-of-topology morphism surjects onto 
 $\pb{\bar{k}_p}\otimes A$. By induction, $\pb{\bar{k}_p}\otimes A$ is already not
finitely generated. Also by  Corollary \ref{cor:pb} we have the
surjection $\pb{k}\otimes A\to\pb{\bar{k}_p}\otimes A$  and the result 
follows.   
 \end{proof}

\begin{remark}
The induction step could also be proved by noting that the conditions
on the field $k$ in the theorem imply that there is a subring
$\mathcal{O}$ of $k$ with the following properties:   
$\mathcal{O}$ is a unique factorization domain with field of fractions
$k$ and for which the set $P$ of association classes of prime elements
is infinite.  
\end{remark}

Finally, we note that in the case of local fields, we can describe the exact structure of the change-of-topology 
kernel over $\Z[1/2]$.

\begin{theorem}
\label{thm:main2c}
Let $k$ be a field satisfying the conditions of  \prettyref{prop:hloc}. Then the
change-of-topology morphism factors through 
$K_3^{\operatorname{ind}}(k)\otimes\mathbb{Z}[1/2]$. Its kernel is isomorphic to 
$$
\bigoplus_{i=0}^{n-1}(\pb{k_i}\otimes\Z[1/2])^{\oplus 2^{n-i-1}}.
$$

In particular, if $k$ is complete with respect to a discrete valuation with residue field $k_0$ which is either 
finite of odd order or real-closed or quadratically closed  
then this kernel is isomorphic to $\pb{k_0}\otimes \Z[1/2]$. 
\end{theorem}
\begin{proof}
Recall that $\mathcal{J}_kH_3(SL_2(k),\Z[1/2])$ is contained in the change-of-topology morphism by Corollary 
\ref{cor:factor}. 
But 
\[
\mathcal{J}_kH_3(SL_2(k),\Z[1/2])=\ker\left(H_3(SL_2(k),\Z[1/2])\to\indk{k}\otimes\Z[1/2]\right) 
\]
by Corollary \ref{cor:hloc}. Since the map to $\indk{k}$ factors through the change-of-topology morphism by Lemma
\ref{lem:bsltoindk}, it follows that $\mathcal{J}_kH_3(SL_2(k),\Z[1/2])$ is equal to the change-of-topology kernel. 

Finally, 
\[
\mathcal{J}_kH_3(SL_2(k),\Z[1/2])  \cong \bigoplus_{i=0}^{n-1}(\pb{k_i}\otimes\Z[1/2])^{\oplus 2^{n-i-1}}
\]
by Corollary \ref{cor:hloc} (3).
\end{proof}

\subsection{Number fields: finite generation}
In the case when $k$ is a number field, we can deduce the failure of weak homotopy invariance in another 
way, which is of independent interest: It follows from
simple size considerations - the group $H_3(SL_2(k),\mathbb{Z})$ is not finitely-generated
while $H_3(BSL_2(k[\Delta^\bullet]),\mathbb{Z})$ is. This last fact is  a consequence of
finite-generation results in symplectic K-theory: 

\begin{proposition}
\label{prop:h3fin}
\begin{enumerate}[(i)]
\item Let $k$ be a non-archimedean local field of characteristic $\neq
  2$, and let $\ell$ be an odd prime different from the
  characteristic. Then the group
  $H_3(BSL_2(k[\Delta^\bullet]),\mathbb{Z}/\ell)$ is finite. 
\item 
Let $k$ be a number field. Then the homology group
$H_3(BSL_2(k[\Delta^\bullet]),\mathbb{Z}[1/2])$ is a finitely
generated $\mathbb{Z}[1/2]$-module. 
\end{enumerate}
\end{proposition}

\begin{proof}
By \prettyref{cor:stabh3}, it suffices to prove the statements for
$H_3(Sp_{\infty}(k),\mathbb{Z})\cong
H_3(BSp_{\infty}(k[\Delta^\bullet]),\mathbb{Z})$. The simplicial set
$BSp_\infty(k[\Delta^\bullet])$ is simply-connected because
$Sp_\infty$ is $\Ao$-connected. Therefore, the Hurewicz theorem
implies a surjection  
$$
\pi_3^{\Ao}(BSp_\infty)(\Spec k)\cong
\pi_3(BSp_\infty(k[\Delta^\bullet]))\rightarrow
H_3(BSp_\infty(k[\Delta^\bullet]),\mathbb{Z}). 
$$
We are thus
reduced to show finite generation for $\pi_3^{\Ao}(BSp_\infty)(k)\cong
KSp_3(k)$.  
The symplectic K-theory assertion is proved in \prettyref{prop:kspfin}
resp. \prettyref{prop:ksploc} below. 
\end{proof}

\begin{proposition}
\label{prop:kspfin}
Let $k$ be a number field. Then
$KSp_3(k)\otimes_{\mathbb{Z}}\mathbb{Z}[1/2]$ is a finitely generated 
$\mathbb{Z}[1/2]$-module. 
\end{proposition}

\begin{proof}
To determine $KSp_3(k)$, we use the computations of Hornbostel,
cf. \cite{hornbostel}. Note that in loc.cit., the group $KSp_3(k)$ is
denoted by $_{-1}K_3^h(k)$. Let $A=\mathcal{O}_{k,S}$ be a ring of
$S$-integers for a finite set $S$ of places containing all the
infinite places and all the places lying above $2$. 
Using \cite[Corollary 4.15]{hornbostel}, we obtain an
exact sequence
$$
\cdots\rightarrow\bigoplus_{\mathfrak{p}}
{}_{-1}U_3(A/\mathfrak{p})\rightarrow {}_{-1}K_3^h(A)\rightarrow 
{}_{-1}K_3^h(k)\rightarrow \bigoplus_{\mathfrak{p}}
{}_{-1}U_2(A/\mathfrak{p})\rightarrow \cdots
$$
The $U$-theory groups of finite fields are determined in
\cite[Corollary 4.17]{hornbostel}:
$$
{}_{-1}U_2(\mathbb{F}_q)\cong\mathbb{Z}/2,\qquad\textrm{
  and }\qquad
{}_{-1}U_3(\mathbb{F}_q)\cong Gr_4,
$$
where $Gr_4$ is either $\mathbb{Z}/4$ or
$\mathbb{Z}/2\oplus \mathbb{Z}/2$. In particular, tensoring the above
exact sequence with $\mathbb{Z}[1/2]$ yields an isomorphism 
${}_{-1}K_3^h(A)\otimes_{\mathbb{Z}}\mathbb{Z}[1/2]\cong
{}_{-1}K_3^h(k)\otimes_{\mathbb{Z}}\mathbb{Z}[1/2]$. 

We are very grateful to the referee for generously providing the following
argument for finite generation of $KSp_3(A)$ for a ring of $S$-integers
$A=\mathcal{O}_{k,S}$. First note that finite generation for algebraic
K-theory is proved in \cite{quillen}. There are Karoubi periodicity
sequences of the form 
$$
K_i(A)\to GW^j_i(A)\to GW^{j-1}_{i-1}(A)\to K_{i-1}(A)
$$
We first use the sequence for $j=2$ and $i=3$, in which case
$GW^2_3(A)=KSp_3(A)$ in our notation:
$$
K_3(A)\to KSp_3(A)\to GW^1_2(A).
$$
By the abovementioned finite generation for algebraic K-theory, it
suffices to show finite generation for $GW^1_2(A)$. Now we use the
Karoubi periodicity sequence for $j=1$ and $i=2$:
$$
K_2(A)\to GW^1_2(A)\to GW^0_1(A),
$$
which reduces us to show finite generation for $GW^0_1(A)=KO_1(A)$. By
\cite[Corollary 4.7.6]{bass}, there is an exact sequence
$$
K_1(A)\to KO_1(A)\to
{}_2\operatorname{Pic}(A)\oplus\operatorname{Hom}(\Spec
A,\mathbb{Z}/2)\to 0.
$$
As we are interested in results with $\mathbb{Z}[1/2]$-coefficients,
the last  term  is $2$-torsion and hence not relevant. Nevertheless, it is
well-known that $S$-class groups are finite which implies finite
generation of the $2$-torsion term. The finite generation for
$KO_1(A)[1/2]$ then follows from finite generation of $K_1(A)$, which
is the well-known finite generation of $S$-unit groups from number
theory. 
\end{proof}

\begin{remark}
Alternatively, one can deduce finite generation of
$KSp_n(\mathcal{O}_{K,S})$ as follows: the existence of the
Borel-Serre compactification of 
$BSp_{2n}(\mathcal{O}_{k,S})$ implies that its homology is finitely
generated. The elementary  subgroup $Ep_{2n}(\mathcal{O}_{k,S})$ is a
perfect finite-index subgroup of
$Sp_{2n}(\mathcal{O}_{k,S})$. Applying Serre's theory of classes of
abelian groups to the universal cover of
$BSp_{\infty}(\mathcal{O}_{k,S})^+$ implies finite generation for
symplectic K-groups. 
\end{remark}

\begin{proposition}
\label{prop:ksploc}
Let $k$ be a non-archimedean local field of characteristic $\neq 2$,
and let $\ell$ be an odd prime different from the characteristic. Then 
$KSp_3(k)\otimes_{\mathbb{Z}}\mathbb{Z}/\ell$ is finite. 
\end{proposition}

\begin{proof}
We denote by $\mathcal{O}$ the valuation ring, and by
$\mathcal{O}/\mathfrak{m}$ the residue field. 
Using the localization sequence for symplectic K-theory as in
\prettyref{prop:kspfin} before, it suffices to prove the assertion for
${}_{-1}U_2(\mathcal{O}/\mathfrak{m})$ and
${}_{-1}K_3^h(\mathcal{O})$. As in \prettyref{prop:kspfin}, 
${}_{-1}U_2(\mathbb{F}_q)$ is a finite $2$-group, hence it does not
contribute. The valuation ring is a complete discrete valuation ring,
therefore we have an isomorphism
$$
{}_{-1}K_3^h(\mathcal{O})\otimes\mathbb{Z}/\ell
\cong{}_{-1}K_3^h(\mathcal{O}/\mathfrak{m})\otimes\mathbb{Z}/\ell. 
$$
The group on the right-hand side is a finite group. 
\end{proof}

\begin{theorem}
\label{thm:main1}
Let $k$ be a number field. Then the kernel of the natural
change-of-topology  morphism 
$$
H_3(SL_2(k),\mathbb{Z})\rightarrow
H_3(BSL_2(k[\Delta^\bullet]),\mathbb{Z})
$$ 
is not finitely generated. 
The same also holds with $\mathbb{Z}/\ell$-coefficients, $\ell$ an odd
prime for which $[k(\zeta_\ell):k]$ is even. 
\end{theorem}

\begin{proof}
Let $A$ denote either $\mathbb{Z}[1/2]$ or $\mathbb{Z}/\ell$ where $\ell$ is an odd prime for which 
$[k(\zeta_\ell):k]$ is even. By the proof of \prettyref{thm:main2b}, the group 
$H_3(SL_2(k),A)$ is not finitely generated while,  by \prettyref{prop:h3fin}, $H_3(BSL_2(k[\Delta^\bullet]),A)$ 
is finitely generated. Thus the kernel of 
$$
H_3(SL_2(k),A)\rightarrow
H_3(BSL_2(k[\Delta^\bullet]),A)
$$ 
is not finitely generated. 
\end{proof}

\section{On the comparison of unstable K-theories}
\label{sec:ktheory}

In this section, we briefly discuss some further consequences of our results
for the comparison of ``unstable K-theories''. There are two possible
definitions of unstable K-theory associated to a linear algebraic
group $G$ and a commutative unital ring $R$ that we want to
discuss. For simplicity, we restrict to $G=SL_n$ and $R=k$, avoiding
subtleties in more general settings.
\begin{itemize}
\item
On the one hand, following
Quillen \cite{quillen:plus}, one can define a version of unstable
K-theory using the plus-construction applied to a classifying space: 
$$
K^Q_i(SL_n,k):=\pi_i BSL_n(k)^+.
$$
The plus-construction here is taken with respect to the maximal
perfect subgroup $E_n(k)=SL_n(k)$. 
\item On the other hand, following Karoubi-Villamayor
\cite{karoubi:villamayor}, one can define a version of unstable
K-theory as homotopy of the classifying space of the simplicial
replacement:
$$
KV_i(SL_n,k):=\pi_i BSL_n(k[\Delta^\bullet]).
$$
\end{itemize}
There is a natural comparison morphism $BSL_n(k)^+\to
BSL_n(k[\Delta^\bullet])$ which is induced from the inclusion of
constants $BSL_n(k)\to BSL_n(k[\Delta^\bullet])$ via the universal
property of the plus-construction. Applying homotopy groups, this
induces a comparison morphism $K^Q_i(SL_n,k)\to KV_i(SL_n,k)$. 

It is well-known (and primarily a consequence of homotopy invariance
of algebraic K-theory) that the comparison morphism becomes a
weak equivalence $BSL(k)^+\to BSL(k[\Delta^\bullet])$ in the limit
$n\to\infty$.  Our results now  show that the comparison morphism
 $BSL_2(k)^+\to BSL_2(k[\Delta^\bullet])$ fails to be a weak
equivalence for large classes of fields.

\begin{proposition}
Let $k$ be a field satisfying the conditions of
\prettyref{thm:main}. Then the natural map 
$$
BSL_2(k)^+\to BSL_2(k[\Delta^\bullet])
$$
is not a weak equivalence. 
\end{proposition}

\begin{proof}
The inclusion of constants $BSL_2(k)\to BSL_2(k[\Delta^\bullet])$
factors through the map in the statement by the universal property of
the plus-construction. Assume that the induced map $BSL_2(k)^+\to
BSL_2(k[\Delta^\bullet])$ is a weak equivalence, in particular it
induces an isomorphism in homology. Since the plus-construction map
$BSL_2(k)\to BSL_2(k)^+$ also induces an isomorphism in homology, this
implies that the change-of-topology map 
$$
H_\bullet(BSL_2(k),\mathbb{Z})\to H_\bullet(BSL_2(k[\Delta^\bullet]),\mathbb{Z})
$$
is an isomorphism. For fields $k$ satisfying the conditions of
\prettyref{thm:main} this yields a contradiction, hence the map
$BSL_2(k)^+\to BSL_2(k[\Delta^\bullet])$ can not be a weak
equivalence. 
\end{proof}

With a slightly more careful analysis, we can even make a stronger
statement, showing that the respective $K_3$-groups can not be isomorphic.

\begin{proposition}
\label{prop:k3a}
Let $k$ be a number field. Then the kernel of the natural map
$K^Q_3(SL_2,k)\to KV_3(SL_2,k)$ is not finitely generated.
\end{proposition}

\begin{proof}
We argue via finite generation as in the proof of
\prettyref{thm:main1}. 

(1) We first note that the finite generation argument of
\prettyref{prop:kspfin} more generally proves that the group
$\pi_3^{\Ao}(BSL_2)(k)\otimes\mathbb{Z}[1/2]$ is 
finitely generated if $k$ is a number field. From \cite[Theorem
3]{asok:fasel}, there is an exact sequence 
$$
0\rightarrow T_4'(k)\rightarrow \pi_2^{\Ao}(SL_2)(\Spec
k)\rightarrow KSp_3(k)\rightarrow 0,
$$
where the last term $KSp_3(k)$ is finitely generated by
\prettyref{prop:kspfin}. 
The group $T_4'(k)$ sits in an exact sequence
$$
I^5(k)\rightarrow T_4'(k)\rightarrow S_4'(k)\rightarrow 0,
$$
where $I^5(k)$ is the fifth power of the fundamental ideal of the Witt
ring $W(k)$ and there is a surjection $K_4^M(k)/12\twoheadrightarrow
S_4'(k)$.

For $k$ be a number field, \cite[Theorem
II.2.1(3)]{bass:tate} implies $K^M_n(k)\cong\mathbb{Z}/2^{r_1}$ where
$n\geq 3$ and $r_1$ is the number of real completions. So $K_4^M(k)$
is finite, hence also $S_4'(k)$ is finite. 

The structure of the Witt ring and the powers of the fundamental ideal
are also explicitly known for number fields, by the Hasse-Minkowski
local-global principle. By \cite[Corollary VI.3.9]{lam:book}, the Witt
ring $W(k)$ splits as $W(k)\cong\mathbb{Z}^r\oplus
W(k)_{\operatorname{tor}}$, where $r$ is the number of real places of
$k$ and the subgroup of torsion elements satisfies
$8(W(k)_{\operatorname{tor}})=0$.  For $n\geq 3$, the $n$-th 
power of the fundamental ideal $I^n(k)=2^{n-1}I(k)$ is free
abelian. In particular, $I^5$ is a free abelian group of rank $r$, in
particular finitely generated. 

We conclude that for $k$ a number field, the group $T_4'(k)$ is
finitely generated. By the exact sequence of Asok and Fasel above, the
group $\pi_2^{\Ao}(SL_2)(k)$ is finitely generated as well, which
seems to be an interesting finite-generation result in
$\Ao$-homotopy. 

(2) Since $SL_2(k)=E_2(k)$, the space $BSL_2(k)^+$ is simply
connected. Therefore, the Hurewicz homomorphism $\pi_3BSL_2(k)^+\to
H_3(BSL_2(k),\mathbb{Z})$ is surjective. As in the proof of
\prettyref{thm:main2b}, the group $H_3(BSL_2(k),\mathbb{Z})$ is not
finitely generated, and therefore $\pi_3BSL_2(k)^+$ is not finitely
generated.

(3) From (1) and (2) above, we conclude that the kernel of the
comparison map $K^Q_3(SL_2,k)\to KV_3(SL_2,k)$ is not finitely generated.
\end{proof}

\begin{remark}
Together 
with the surjection 
$$\pi_3^{\Ao}(BSL_2)(k)\otimes\mathbb{Z}[1/2]\to
H_3(BSL_2(k[\Delta^\bullet]),\mathbb{Z}[1/2]),
$$
the above argument provides another way of proving 
\prettyref{thm:main1} without resorting to the improved stability
results of \prettyref{cor:stabh3}. 
\end{remark}

\begin{proposition}
\label{prop:k3b}
Assume that $k$ is a field satisfying the conditions of
\prettyref{thm:main}. Then the comparison map $K^Q_3(SL_2,k)\to
KV_3(SL_2,k)$ is not  injective.
\end{proposition}

\begin{proof}
We saw earlier in \prettyref{sec:mod} that the $k^\times$-action on
$BSL_2(k[\Delta^\bullet])$ induces a 
$\mathbb{Z}[k^\times/(k^\times)^2]$-module structure on
$\pi_3BSL_2(k[\Delta^\bullet])$ which descends to a $GW(k)$-module
structure, by \prettyref{prop:gw}. 
As in \prettyref{sec:mod}, the
conjugation with diagonal matrices $\operatorname{diag}(u,1)$
also induces a $\mathbb{Z}[k^\times/(k^\times)^2]$-module structure on
$\pi_3BSL_2(k)^+$, and the comparison map 
$$\pi_3BSL_2(k)^+\to
\pi_3BSL_2(k[\Delta^\bullet])$$
is $\mathbb{Z}[k^\times/(k^\times)^2]$-equivariant. We have a
commutative diagram
\begin{center}
\begin{minipage}[c]{10cm}
\xymatrix{
\ker H\ar[r]\ar[d] & \ker H\otimes GW(k)\ar[d] \\
\pi_3BSL_2(k)^+\ar[r]\ar[d]_H & \pi_3BSL_2(k)^+\otimes GW(k)\ar[d] \\ 
H_3(BSL_2(k),\mathbb{Z})\ar[r] \ar[d] &
H_3(BSL_2(k),\mathbb{Z})\otimes GW(k)\ar[d]  \\
0&0
}
\end{minipage}
\end{center}
We omitted that all tensors above are over
$\mathbb{Z}[k^\times/(k^\times)^2]$ for typographical reasons.
The map $\ker H\to\ker H\otimes GW(k)$ is surjective since
$\mathbb{Z}[k^\times/(k^\times)^2]\to GW(k)$ is surjective. Assuming
$\pi_3BSL_2(k)^+$ is a $GW(k)$-module, then $\pi_3BSL_2(k)^+\to
\pi_3BSL_2(k)^+\otimes GW(k)$ is an isomorphism. But  then $\ker H$ and
$\ker H\otimes GW(k)$ have isomorphic images in $\pi_3BSL_2(k)^+$ and
$\pi_3BSL_2(k)^+\otimes GW(k)$, respectively. This implies that
$H_3(BSL_2(k),\mathbb{Z})$ already has a $GW(k)$-module structure,
contradicting the conclusion of \prettyref{thm:main}. 
\end{proof}

The above results imply that there is no unique natural definition of 
unstable K-theory. 

\section{The cokernel of the change-of-topology morphism}
\label{sec:cokernel}
The main results in this article concern estimates of the kernel of
the change-of-topology morphism from $H_3(SL_2(k), A)$ to
$H_3(BSL_2(k[\Delta^\bullet]), A)$. In this section, we discuss  the
cokernel of this morphism.   
In order to do this, we treat some aspects of the spectral sequence
associated to the bisimplicial set $BSL_2(k[\Delta^\bullet])$. This
spectral sequence computes the homology of the diagonal simplicial set
$dBSL_2(k[\Delta^\bullet])$ and has the form
$$
E^1_{p,q}=H_q(BSL_2(k[\Delta^p])),\mathbb{Z})
\Rightarrow H_{p+q}(dBSL_2(k[\Delta^\bullet]),\mathbb{Z}),
$$
with differentials $d^r_{p,q}:E^r_{p,q}\rightarrow E^r_{p-r,q+r-1}$.
Here, the differentials
\[
d^1_{p,q}=\sum_{i=0}^p(-1)^iH_q(d_i):H_q(BSL_2(k[\Delta^p])),\mathbb{Z})\to H_q(BSL_2(k[\Delta^{p-1}])),\mathbb{Z})
\]
 are induced from the simplicial structure of the
simplicial algebra $k[\Delta^\bullet]$. 

\emph{We will assume in this section that the underlying field $k$ is
  infinite of characteristic $\neq 2$.}

\begin{remark}
\label{rem:van}
For infinite $k$, homotopy invariance for homology of $SL_2(k[T])$,
cf. \cite[Theorem 4.3.1]{knudson:book}, implies $E^2_{1,q}\cong 0$ for
all $q>1$. 

Since $H_0(SL_2(k[\Delta^n]),\mathbb{Z})\cong\mathbb{Z}$ for all $n$,
we also have
$$
E^2_{p,0}\cong\left\{\begin{array}{ll}
\mathbb{Z} & p=0\\
0 & \textrm{otherwise}
\end{array}\right.
$$
\end{remark}

\begin{remark}
\label{rem:d221}
The differential $d^2_{2,1}$ is trivial: by the stabilization results,
there is a natural commutative diagram
\begin{center}
\begin{minipage}[c]{10cm}
\xymatrix{
H_2(SL_2(k)) \ar[r] \ar[d]_\cong & H_2(BSL_2(k[\Delta^\bullet]))
\ar[d]^\cong \\ 
H_2(Sp_\infty(k))\ar[r]_\cong & H_2(BSp_\infty(k[\Delta^\bullet])) 
}
\end{minipage}
\end{center}
The right vertical isomorphism is a consequence of
\prettyref{prop:stabsp}, the left vertical isomorphism is a consequence
of Matsumoto's computation of $H_2$. The lower isomorphism follows
from homotopy invariance for the infinite symplectic group
\cite{karoubi}. By the vanishing in \prettyref{rem:van}, we have an 
exact sequence  
$$
H_1(SL_2(k[\Delta^2]),\mathbb{Z})/(\operatorname{im}
d^1_{3,1})\stackrel{d^2_{2,1}}{\longrightarrow}
H_2(SL_2(k),\mathbb{Z})\rightarrow
H_2(BSL_2(k[\Delta^\bullet]),\mathbb{Z})\rightarrow 0,
$$
which proves the claim.
\end{remark}

\begin{proposition}
\label{prop:result}
There is a short exact sequence
$$
0\rightarrow H_3(SL_2(k))/(\operatorname{im}
d^2_{2,2}+\operatorname{im} d^3_{3,1})\rightarrow
H_3(BSL_2(k[\Delta^\bullet]))\rightarrow
$$
$$\rightarrow
H_1(SL_2(k[\Delta^2]))/(\operatorname{im} d^1_{3,1})\rightarrow 0
$$
Moreover, the natural change-of-topology $H_3(SL_2(k))\rightarrow
H_3(BSL_2(k[\Delta^\bullet]))$ factors through the injection above.
\end{proposition}

\begin{proof}
By the vanishing in \prettyref{rem:van}, we have
$E^\infty_{1,2}=E^\infty_{3,0}=0$. The exact sequence claimed is the
one induced from the spectral sequence:
$$
0\rightarrow E^\infty_{0,3}\rightarrow
H_3(BSL_2(k[\Delta^\bullet]))\rightarrow E^\infty_{2,1}\rightarrow 0.
$$

We prove  $E^\infty_{2,1}=H_1(SL_2(k[\Delta^2]))/(\im
d^1_{3,1})$. We have  $H_1(SL_2(k[T]))=0$, since $SL_2(k[T])$ is
perfect. Hence the differential $d^1_{2,1}$ is trivial. No
differential except $d^1_{3,1}$ hits the $(2,1)$-entry. By
\prettyref{rem:d221}, the differential $d^2_{2,1}$ is also
trivial. This proves the claim. 

We identify $E^\infty_{0,3}$. All differentials starting at $(0,3)$
are trivial. Therefore, $E^\infty_{0,3}$ is the quotient of
$H_3(SL_2(k))$ by all differentials hitting it. The differentials
$d^1_{1,3}$  and $d^1_{4,0}$ are trivial by \prettyref{rem:van}. Only
the differentials $d^2_{2,2}$ and $d^2_{3,1}$ remain.

The last statement is obvious, since the  natural change-of-topology
morphism includes $SL_2(k)$ as $0$-simplices. Therefore, the natural
map factors as 
$$
H_\bullet(SL_2(k),\mathbb{Z})\rightarrow
E^\infty_{0,\bullet}\rightarrow
H_\bullet(BSL_2(k[\Delta^\bullet]),\mathbb{Z}).
$$
\end{proof}

\begin{remark}
This is a good place to point out that the above result implies that
the spectral sequence 
$$
E^1_{p,q}=H_q(BSL_2(k[\Delta^p])),\mathbb{Z})
\Rightarrow H_{p+q}(dBSL_2(k[\Delta^\bullet]),\mathbb{Z}),
$$
does not degenerate at the $E^2$-page. Moreover, the differentials
$d^2_{2,2}$ and $d^3_{3,1}$ provide a natural relation between the
counterexamples to homotopy invariance for homology of $SL_2$ and the
kernel of the natural change-of-topology morphism 
$$
H_3(SL_2(k),\mathbb{Z})\rightarrow
H_3(BSL_2(k[\Delta^\bullet]),\mathbb{Z}). 
$$
In the cases discussed in \prettyref{thm:main2b} and
\prettyref{thm:main2c}, the above differentials induce non-trivial
morphisms from homotopy-invariance-counterexamples to pre-Bloch groups
of residue fields. 

The explicit computation of such differentials is very
complicated, for various reasons. First of all, our knowledge of the
groups $H_2(SL_2(k[\Delta^2]))$ and $H_1(SL_2(k[\Delta^3]))$ is very
limited. The constructions of \cite{krstic:mccool} show that these
groups tend to be very large, but do not give a precise description of
their structure. Second, the essential step in the computation of
$d^2_{2,2}$ needs explicit lifts of null-homologous cycles in
$\mathcal{Z}_2(SL_2(k[T]))$ to $3$-chains. While the amalgam
decomposition of $SL_2(k[T])$ can in principle be used to compute such
things, the computations easily become too complicated to follow
through. 
\end{remark}

\begin{proposition}
Let $k$ be an infinite field of characteristic $\neq 2$. 
There is a commutative diagram with exact columns
\begin{center}
\begin{minipage}[c]{10cm}
\xymatrix{
  0 \ar[d] & 0 \ar[d] \\
  H_3(SL_2(k))/(\im d^2_{2,2}+\im d^3_{3,1}) \ar[d] \ar[r] & 
  H_3(Sp_\infty(k)) \ar[d] \\
  H_3(BSL_2(k[\Delta^\bullet]))\ar[r]^s \ar[d] & 
  H_3(BSp_\infty(k[\Delta^\bullet])) \ar[d] \\
  E_{2,1}^\infty \ar[d] \ar[r] & 0\\ 0
}
\end{minipage}
\end{center}
The map $s$ above is surjective. 
In particular, we have a surjection
$$
t:E^\infty_{2,1}\twoheadrightarrow
\operatorname{coker}\left(H_3(SL_2(k))\rightarrow 
  H_3(Sp_\infty(k))\right). 
$$
If $k$ is of characteristic $0$, then $s$ and $t$ are isomorphisms. 
\end{proposition}

\begin{proof}
We first note  that the inclusions of groups $SL_2\hookrightarrow Sp_\infty$
induces a morphism of the  bisimplicial object
$BSL_2(k[\Delta^\bullet])\rightarrow
BSp_\infty(k[\Delta^\bullet])$. The spectral sequence computing the
homology of the diagonal of a bisimplicial object is compatible with
morphisms of bisimplicial objects.  The exact column on the left is a
consequence of \prettyref{prop:result}. By \cite{karoubi}, the
homology of the infinite symplectic group has $\Ao$-invariance for
regular rings in which $2$ is invertible. Therefore the spectral
sequence associated to the bisimplicial object
$BSp_\infty(k[\Delta^\bullet])$ collapses and produces the exact
column on the right. The whole diagram is commutative by the
abovementioned compatibility of the spectral sequences with the
stabilization morphism. 

Surjectivity of $s$ is a consequence of stabilization
\prettyref{prop:stabsp}, and the characteristic $0$ isomorphism is a
consequence of our improved stabilization result
\prettyref{cor:stabh3}. An application of the 
snake-lemma then proves the last claim. Note in particular that the
top horizontal morphism is induced from the standard inclusion
$SL_2\rightarrow Sp_\infty$, hence it is really the stabilization
morphism.
\end{proof}

\begin{remark}
It is interesting to see that the surjectivity of the
change-of-topology map can be completely translated into a question on
stabilization of the homology of linear groups, namely the question of
surjective stabilization for the morphism $H_3(SL_2)\rightarrow
H_3(Sp_\infty)$.  
The known stabilization results for the symplectic groups do not seem
to decide surjective stabilization in this range.
\end{remark}

\end{document}